\documentclass[12pt]{amsart}
\usepackage{amsmath,amssymb,amsthm,amscd}

\usepackage{graphicx}
\graphicspath{}
\DeclareGraphicsExtensions{.pdf,.png,.jpg,.jpeg}

\newtheorem{formula}{}[section]
\newtheorem{proposition}[formula]{Proposition}
\newtheorem{corollary}[formula]{Corollary}

\newtheorem{theorem}[formula]{Theorem}
\theoremstyle{definition}
\newtheorem{definition}[formula]{Definition}
\newtheorem{example}[formula]{Example}
\theoremstyle{remark}
\newtheorem*{remark}{Remark}

\DeclareMathOperator{\cc}{cc}
\newcommand{\zp}{\mathcal Z_P}

\newcommand{\zk}{\mathcal Z_K}

\newcommand{\mb}[1]{{\textbf {\textit#1}}}

\begin{document}

\title{Higher order Massey products and applications}
\author[I.Yu.~Limonchenko]{Ivan~Limonchenko}
\address{Faculty of Computer Science, National Research University Higher School of Economics, Moscow, Russia}
\email{ilimonchenko@hse.ru}
\address{The Fields Institute for Research in Mathematical Sciences, Department of Mathematics, University of Toronto, Toronto, Canada}
\email{ilimonch@math.toronto.edu}
\author[D.V.~Millionshchikov]{Dmitry~Millionshchikov}

\address{Department of Mechanics and Mathematics, Moscow State University, 119992 Moscow, Russia}
\email{mitia\_m@hotmail.com}
\address{Steklov Mathematical Institute of Russian Academy of Sciences,  8 Gubkina St., Moscow 119991, Russia}
\address{Gubkin Russian State University of Oil and Gas (National Research 
University), 65 Leninsky Prospekt, 119991 Moscow,  Russia}

\def\sgn{\mathrm{sgn}\,}
\def\bideg{\mathrm{bideg}\,}
\def\tdeg{\mathrm{tdeg}\,}
\def\sdeg{\mathrm{sdeg}\,}
\def\grad{\mathrm{grad}\,}
\def\ch{\mathrm{ch}\,}
\def\sh{\mathrm{sh}\,}
\def\th{\mathrm{th}\,}

\def\mod{\mathrm{mod}\,}
\def\In{\mathrm{In}\,}
\def\Im{\mathrm{Im}\,}
\def\Ker{\mathrm{Ker}\,}
\def\Hom{\mathrm{Hom}\,}
\def\Tor{\mathrm{Tor}\,}
\def\rk{\mathrm{rk}\,}
\def\codim{\mathrm{codim}\,}

\def\ko{{\mathbf k}}
\def\sk{\mathrm{sk}\,}
\def\RC{\mathrm{RC}\,}
\def\gr{\mathrm{gr}\,}

\def\R{{\mathbb R}}
\def\C{{\mathbb C}}
\def\Z{{\mathbb Z}}
\def\A{{\mathcal A}}
\def\B{{\mathcal B}}
\def\K{{\mathcal K}}
\def\M{{\mathcal M}}
\def\N{{\mathcal N}}
\def\E{{\mathcal E}}
\def\G{{\mathcal G}}
\def\D{{\mathcal D}}
\def\F{{\mathcal F}}
\def\L{{\mathcal L}}
\def\V{{\mathcal V}}
\def\H{{\mathcal H}}

\address{Department of Mechanics and Mathematics, \\
Moscow State University,\\
Leninskie Gory, 119992 Moscow, Russia \vspace{1mm}\\
E-mail: mitia\_m@hotmail.com }

\thanks{The research of the first author was carried out within the University Basic Research Program of the Higher School of Economics and was funded by the Russian Academic Excellence Project `5-100'. The second author was supported by the RSF grant 20-11-19998. The work of D.V.~Millionshchikov (the second author) is supported by the Russian Science Foundation under grant 20-11-19998 and performed in Steklov Mathematical Institute of Russian Academy of Sciences.}



\maketitle

\begin{abstract}
{In this survey, we discuss two research areas related to Massey's higher operations. The first direction is connected with the cohomology of Lie algebras and the theory of representations. The second main theme is at the intersection of toric topology, homotopy theory of polyhedral products, and the homology theory of local rings, Stanley--Reisner rings of simplicial complexes.}
\end{abstract}

\setcounter{section}{0}

\section*{Introduction}
Higher order Massey operations proved to be a very effective algebraic tool for describing various obstructions to the existence or continuation of the diverse topological, geometric or algebraic structures and their deformations.
In this short review, we attempted to systematize a series of results obtained over the past decade concerned with two important topics related to applications of the Massey higher operations. The starting point for the first topic is the Duady's work \cite{Douady} of the 1960, where the relation of Massey products to the theory of deformations was observed. This topic has been actively developed in the following decades, we especially note the works by Palamodov \cite{Palam}, Retakh \cite{Retakh1, Retakh2}, in which the connection between the higher Massey products and the Kodaira-Spencer theory of deformations was studied. 

In algebraic topology it is well known that Massey products serve as an obstruction to formality of a topological space (see Section 5).  
In 1975, Deligne, Griffiths, Morgan and Sullivan proved that simply connected compact K\"ahler manifolds are formal \cite{DGMS}. In particular it means that the existence of non-trivial Massey products in the cohomolgy $H^*(M,{\mathbb R})$ is an obstruction for a manifold $M$ to be K\"ahler \cite{DGMS}. 
On the other hand, Halperin and Stasheff~\cite{HS} constructed a non-formal differential graded algebra such that all Massey products vanish in its cohomology.

A blow-up of a symplectic manifold $M$ along its submanifold $N$
inherits non-trivial Massey products \cite{BaTa0, BaTa1}.
Babenko and Taimanov  applied the symplectic blow-up procedure for the construction of simply
connected non-formal symplectic manifolds in dimensions $\ge 10$ \cite{BaTa0, BaTa1}. 
However, in the papers of Babenko and Taimanov \cite{BaTa1, BaTa2} we were attracted primarily by their approach to the definition of Massey products in the language of formal connections and the Maurer-Cartan equation. The idea of such an approach has been encountered in literature before, first  May's \cite{JPM, JPM2}, then in Palamodov's article \cite{Palam}, but it was precisely in the articles of Babenko and Taimanov that this approach was comprehensively developed \cite{BaTa1, BaTa2}. 

In the Section \ref{Massey_S} we
recall the elements of the Babenko-Taimanov approach to the definition of Massey products \cite{BaTa1, BaTa2}.  The analogy with the classical
Maurer-Cartan equation, which is especially transparent in the case of
Massey products of $1$-dimensional cohomology classes $\langle
\omega_1,\dots,\omega_n\rangle$, is discussed in Section \ref{g-modules}. The relation of this special case to representation theory was discovered in \cite{Dw, FeFuRe}. 

The related material, which is presented in the form of a publication for the first time, is contained in Section \ref{k-step}. We are talking about variations of Massey products, which we called $k$-step Massey products. It is well known that the classical Massey products are multi-valued and partially defined operations. We propose to consider successive obstructions $\langle \omega_1, \dots, \omega_n \rangle_k, k=1, \dots, n{-}1,$ arising in constructing a formal connection (defining system) $A$ for the classical Massey product $\langle \omega_1,\dots, \omega_n \rangle$ as $k$-step Massey products (see Definition \ref{def-k-step}).

The main feature uniting the results that we relate to the first topic is the Massey products in the cohomology of Lie algebras. Particular attention in this part is given to non-trivial Massey products, it was motivated by applications. Two very important and interesting positively graded Lie algebras were chosen as the main examples of this article. Firstly, this is the positive part $W^+$ of the Witt algebra, and secondly, its associated graded algebra   with respect to the filtration by the ideals of the lower central series ${\mathfrak m}_0$. Sometimes ${\mathfrak m}_0$ is called 
the infinite dimensional filiform Lie algebra. We discuss in Section \ref{nontrivSection} the proof of Buchstaber's conjecture that the cohomology $H^*(W^+)$ are generated by means of non-trivial Massey products by the one-cohomology $H^1(W^+)$. The corresponding Theorem \ref{main_Buch_conj} was proved by the second author in \cite{Mill2}. We consider in Section \ref{nontrivSection} also
the structure results on the Massey products in the cohomology  $H^*({\mathfrak m}_0)$.

There is a natural connection between the cohomology of infinite-dimensional positively graded (filtered) Lie algebras and the topology of manifolds. The cohomology of finite-dimensional quotients of such algebras with rational structural constants is isomorphic, according to the Nomizu theorem \cite{Nomz}, to the real cohomology of nilmanifolds $G/\Gamma$ that correspond to such quotients. It was such a tower of nilmanifolds $M_n$, that correspond  to $W^+$ (the Witt algebra) and first introduced by Buchstaber in \cite{BuSho} and later used by Babenko and Taimanov in their construction of symplectic nilmanifolds \cite{BaTa0, BaTa1} with simply defined non-trivial triple Massey products. Namely, the methods of the Lie algebras theory  make it possible to efficiently calculate the first cohomology of nilmanifolds and discover nontrivial Massey products.

A well-known conjecture, which dates back to the papers by May~\cite{JPM} and May-Gugenheim~\cite{JPMG}, asserts that higher differentials in the Eilenberg--Moore spectral sequence of a space are determined by (matric) Massey products of the space. The second central theme of our review is closely related to that conjecture and originates in the early 1960s, from a pioneering paper of Golod~\cite{Go}. He showed that Poincar\'e series of a local ring $A$ achieve its (coefficientwise) upper bound, previously identified by Serre, precisely in the case when multiplication and all Massey products, triple and higher, vanish in Koszul homology of $A$. 

This result acquired a topological interpretation in toric topology by means of the theorem due to Buchstaber and Panov~\cite{bu-pa00-2,TT}, who proved that cohomology algebra of a moment-angle-complex $\zk$ over a commutative ring with unit $\ko$ is isomorphic to Koszul homology of the corresponding Stanley-Reisner ring $\ko[K]$. It gives us a tool to identify a class of simplicial complexes with formal moment-angle complexes and it also enables one to construct non-formal (and therefore, non-K\"ahler) moment-angle manifolds, having non-trivial Massey products in their cohomology. 

Another classical well-known problem related to this part of our survey is the Steenrod problem of realization of (rational) cycles in a given space by oriented manifolds. Halperin and Stasheff~\cite{HS} showed that cohomology ring of a (rationally) formal space is generated by spherical classes. Thanks to toric topology, an example of a non-trivial triple Massey product of 3-dimensional spherical classes in cohomology of a 2-connected space (a moment-angle-complex) was contructed by Baskakov~\cite{BaskM}. In~\cite{L2016,L2017,L2019} the first author generalized Baskakov's construction and proved existence of moment-angle manifolds $\zp$ having non-trivial higher Massey products (of any prescribed order) of spherical cohomology classes in $H^*(\zp)$.

In Section 5 we give a survey on the construction of Massey products in the algebraic context of Koszul homology of local rings and discuss the Golod property for Stanley--Reisner rings of simplicial complexes. In Section 6 we deal with the results on non-trivial triple and higher Massey products in cohomology of moment-angle-complexes and moment-angle manifolds, emphasizing the case of strictly defined (i.e., containing a single element) non-trivial higher Massey products.


\section{Massey products in cohomology}
\label{Massey_S}
Let $\A=\oplus_{l \ge 0}\A^l$
be a differential graded algebra over a field ${\mathbb K}$.
It means that the following operations are defined:
an associative multiplication
$$
\wedge:\A^l\times\A^m\to\A^{l+m},\; l,m\geq 0, \; l,n \in {\mathbb Z}.
$$
such that $a\wedge b=(-1)^{lm}\,b\wedge a$ for $a\in\A^l$, $b\in\A^m$,
and a differential $d, \; d^2=0$
$$
d:\A^l\to\A^{l+1},\ \ l\geq 0,
$$
satisfying the Leibniz rule
$d\,(a\wedge b)=d\,a\wedge b+(-1)^l a\wedge d\,b$ for $a\in\A^l$.

Of course, the most natural example of differential graded algebra $\A$ is the de Rham complex $\A=\Lambda^*(M), {\mathbb K}={\mathbb R},$ of smooth forms of a smooth manifold $M$. However, in this article we will pay special attention to the following two examples.

\begin{example}
$\A=\Lambda^*(\mathfrak{g})$ is the cochain complex of a Lie algebra.
\end{example}

\begin{example}
A Koszul complex $K_{A}=\Lambda\,A^{m}$ of a commutative Noetherian local ring $(A,\mathbf m,\ko)$ (see the section \ref{local_rings} for details).
\end{example}

For a given differential graded algebra  $(\A,d)$ we denote
by $T_n(\A)$ a space of all upper triangular
$(n+1)\times (n+1)$-matrices with entries from $\A$,
vanishing at the main diagonal. The standard matrix multiplication turns the vector space $T_n(\A)$ into an algebra, we assume that
matrix entries are multiplying as elements of $\A$.
One can define the differential $d$ on $T_n(\A)$  by
\begin{equation}
d\,A=(d\,a_{ij})_{1\leq i,j\leq n{+}1}.
\label{1.2.1}
\end{equation}

We extend the involution $a\to\bar{a}=(-1)^{k+1}a, a \in A^k$
of $\A$  to the involution of $T_n(\A)$ by the rule
$$
\bar{A}=(\bar{a}_{ij})_{1\leq i,j\leq n{+}1}.
$$
It satisfies the following properties
$$
\overline{\bar{A}}=A, \quad
\overline{AB}=-\bar{A}\bar{B}, \quad \overline{d\,A}=-d\,\bar{A}.
$$
Also we have the generalized Leibniz rule for the differential
(\ref{1.2.1})
$$
d\,(AB)=(d\,A)B-\bar{A}(d\,B).
$$

Consider a
two-sided ideal $I_n(\A)$ of matrices of the following form
$$
\left(\begin{array}{cccc}
0 & \dots & 0 & \tau \\
0 & \dots & 0 & 0 \\
& \dots& & \\
0 & \dots & 0 & 0
\end{array}\right), \quad \tau \in \A
$$
Obviously, the ideal $I_n(\A)$ belongs to the center $Z(T_n(\A))$ of the algebra $T_n(\A)$.

\begin{definition}[\cite{BaTa2}]
A matrix $A \in T_n(\A)$ is called the matrix of
a formal connection if it satisfies the Maurer-Cartan equation
\begin{equation}
\mu(A)=d\,A-\bar{A}\cdot A \in I_n(\A).
\label{star}
\end{equation}
\end{definition}
\begin{proposition}[\cite{BaTa2}]
The generalized Bianchi identity
for the Maurer-Cartan operator $\mu(A)=d\,A-\bar{A}\cdot A$ holds
\begin{equation}
\label{MCar}
d\,\mu(A)=\overline{\mu(A)}\cdot A+A\cdot\mu(A).
\end{equation}
\end{proposition}
\begin{proof}
Indeed it's easy to verify the following equalities
$$
\begin{array}{c}
d\,\mu(A)=-d\,(\bar{A}\cdot A)=
-d\,\bar{A}\cdot A+A\cdot d\,A=\overline{d\,A}\cdot A+A\cdot d\,A=\\
=\overline{(\mu(A)+\bar{A}\cdot A)}\cdot A+A(\mu(A)+\bar{A}\cdot A)
=\overline{\mu(A)}\cdot A+A\cdot\mu(A).
\end{array}
$$
\end{proof}
\begin{corollary}[\cite{BaTa2}]
Let $A$ be the matrix of a formal connection,
then the entry $\tau \in \A$ of the matrix $\mu(A) \in I_n(\A)$
in the definition (\ref{star}) is closed.
\end{corollary}

Now let $A$ be the matrix of a formal connection, then
the matrix $\mu(A)$ belongs to the ideal $I_n(\A)$ and hence $d\mu(A)=0$.
In a formal sense $\mu(A)$ plays the role of the curvature
matrix of a formal connection $A$.

Let $A$ be an upper triangular matrix from $T_n(\A)$.
$$
A=\left(\begin{array}{cccccc}
  0      & a(1,1) & a(1,2) & \dots & a(1,n-1)   & a(1,n)   \\
  0      & 0      & a(2,2) & \dots & a(2,n-1)   & a(2,n)   \\
  \dots  & \dots  & \dots  & \dots & \dots      & \dots    \\
  0      & 0      & 0      & \dots & a(n-1,n-1) & a(n-1,n) \\
  0      & 0      & 0      & \dots & 0          & a(n,n)   \\
  0      & 0      & 0      & \dots & 0          & 0
  \end{array}\right).
$$

\begin{proposition}
A matrix $A \in T_n(\A)$ is the matrix of a formal connection if and only if
the following equalities hold
\begin{equation}
\label{def_syst}
\begin{split}
a(i,i)=a_i \in \A^{p_i}, \quad i=1,\dots,n;\\
a(i,j)\in\A^{p(i,j)+1}, \quad p(i,j)=\sum^j_{r=i}(p_r-1);\\
d\,a(i,j)=\sum_{r=i}^{j-1}\bar{a}(i,r)\wedge a(r+1,j),\;\;
(i,j)\neq(1,n).
\end{split}
\end{equation}
\end{proposition}
The system (\ref{def_syst}) is  just the Maurer-Cartan equation
rewritten in terms of the entries of the matrix $A$ and it
is a part of the classical definition \cite{K} of
the defining system for a Massey product.
\begin{definition}[\cite{K}]
A collection of elements,
$A=(a(i,j))$, for $1\leq i\leq j\leq n$ and $(i,j)\neq(1,n)$
is said to be a defining  system for the product
$\langle a_1,\dots,a_n\rangle$ if it satisfies (\ref{def_syst}).

Under these conditions
the $(p(1,n)+2)$-dimensional cocycle
$$
c(A)=\sum_{r=1}^{n-1}\bar{a}(1,r)\wedge a(r+1,n)
$$
is called the related cocycle of the defining system $A$.
\end{definition}

One can verify that the notion of the defining system is equivalent to
the notion of the formal connection. We have only to remark
that an entry $a(1,n)$ of a formal connection  $A$
does not belong to the corresponding defining system.
It can be taken as an arbitrary element from $\A$ and
for the only one nonzero (possibly) entry
$\tau \in \mu(A)$ we have 
$$
\tau=-c(A)+da(1,n).
$$

\begin{definition}[\cite{K}]
The $n$-fold product $\langle a_1,\dots,a_n\rangle$ is defined if
there exists at least one defining system for it
(a formal connection $A$ with
entries $a_1,\dots,a_n$ at the second diagonal). If it is defined, then
the value $\langle a_1,\dots,a_n\rangle$ is the set of all cohomology
classes $\alpha \in H^{p(1,n)+2}(\A)$ for which
there exists a defining system $A$
such that the cocycle $c(A)$ (or equivalently $-\tau$) represents $\alpha$.
\end{definition}

\begin{theorem}[see \cite{K},\cite{BaTa2}]
The product  $\langle a_1,\dots,a_n\rangle$ depends only on the
cohomology classes of the elements $a_1,\dots,a_n$.
\end{theorem}

\begin{definition}[\cite{K}]
A set of closed elements $a_i, i=1,\dots,n$ from $\A$ representing
some cohomology classes ${\alpha}_i \in H^{p_i}(\A), i=1,\dots,n$
is said to be a defining system for
the Massey $n$-fold product $\langle {\alpha}_1,\dots,{\alpha}_n\rangle$
if it is one for $\langle a_1,\dots,a_n\rangle$. The Massey
$n$-fold product $\langle {\alpha}_1,\dots,{\alpha}_n\rangle$
is defined if $\langle a_1,\dots,a_n\rangle$ is defined, in which case
$\langle {\alpha}_1,\dots,{\alpha}_n\rangle=\langle a_1,\dots,a_n\rangle$
as subsets in $H^{p(1,n)+2}(\A)$.
\end{definition}

For $n=2$ the matrix $A$ of a formal connection is
$$
A=\left(\begin{array}{ccc}
0 & a & c\\
0 & 0 & b\\
0& 0& 0
\end{array}\right)
$$
and the generalized Maurer-Cartan equation is equivalent
to the system $da=0, \; db=0$. Hence a $2$-fold Massey product $\langle \alpha, \beta \rangle$ is always defined and 
$\langle \alpha, \beta \rangle = \bar \alpha \wedge \beta$.

Let $\alpha$, $\beta$, and  $\gamma$ be the cohomology classes of closed
elements $a \in\A^p$, $b\in\A^q$, and  $c\in\A^r$.
The Maurer-Cartan equation for
$$
A=\left(\begin{array}{cccc}
0 & a & f & h \\
0 & 0 & b & g \\
0 & 0 & 0 & c \\
0 & 0 & 0 & 0
\end{array}\right).
$$
is equivalent to the system
\begin{equation}
d\,f=(-1)^{p+1}\,a\wedge b,\ \ d\,g=(-1)^{q+1}\,b\wedge c.
\label{ast}
\end{equation}
Hence the triple Massey product
$\langle\alpha,\beta,\gamma\rangle$
is defined if and only if the following conditions hold
$$
\alpha\cdot \beta=\beta \cdot \gamma=0\ \ \mbox{in}\ \ H^{\ast}(\A).
$$

The triple Massey product $\langle \alpha, \beta, \gamma \rangle$ is defined
as a subspace $H^{p{+}q{+}{r}{-}1}(\A)$ of elements
$$
\langle \alpha, \beta, \gamma \rangle=\left\{
[(-1)^{p+1} a\wedge g+(-1)^{p+q} f\wedge c]\right\}.
$$
Since $f$ and  $g$ are defined by (\ref{ast}) up to closed elements
from $\A^{p+q-1}$  and $\A^{q+r-1}$ respectively, the triple Massey product
$\langle \alpha,\beta,\gamma\rangle$ is an affine subspace of
$H^{p{+}q{+}{r}{-}1}(\A)$ parallel to
$\alpha \cdot H^{q+r-1}(\A)+ H^{p+q-1}(\A)\cdot \gamma$.

Sometimes the triple Massey product  $\langle \alpha, \beta, \gamma \rangle$ is defined as a quotient 
$\langle \alpha, \beta, \gamma \rangle/
(\alpha \cdot H^{q+r-1}(\A)+ H^{p+q-1}(\A)\cdot \gamma)$ \cite{Fu}.

\begin{definition}
Let an  $n$-fold Massey product
$\langle {\alpha}_1,\dots,{\alpha}_n\rangle$
be defined. It is called trivial if it contains the trivial cohomology class:
$0\in\langle {\alpha}_1,\dots,{\alpha}_n\rangle$.
\end{definition}

\begin{proposition}
\label{triviality}
Let a Massey product
$\langle {\alpha}_1,\dots,{\alpha}_n\rangle$
is defined. Then all Massey products
$\langle {\alpha}_l,\dots,{\alpha}_q\rangle, 1\le l < q \le n, q-l<n-1$
are defined and trivial.
\end{proposition}

The triviality of all Massey products
$\langle {\alpha}_l,\dots,{\alpha}_q\rangle, 1\le l < q \le n, q-l<n-1$
is only a necessary condition
for a Massey product $\langle {\alpha}_1,\dots,{\alpha}_n\rangle$
to be defined. It is sufficient only in the case $n=3$.

Let us denote by $GT_n({\mathbb K})$ a group of non-degenerate
upper triangular $(n{+}1,n{+}1)$-matrices of the form:
 \begin{equation}
\label{GT_n}
C=\left(\begin{array}{ccccc}
  c_{1,1}      & c_{1,2} & \dots & c_{1,n}   & c_{1,n{+}1}   \\
  0      & c_{2,2}      &  \dots & c_{2,n}  & c_{2,n{+}1}   \\
 \dots   &        &       \dots &            & \dots     \\
  0      & 0            & \dots & c_{n,n}& c_{n,n{+}1}   \\
  0      & 0            & \dots & 0          & c_{n{+}1,n{+}1}
  \end{array}\right).
\end{equation}
\begin{proposition}
Let $A \in T_n({\A})$ be the matrix of a formal connection and
$C$ an arbitrary matrix from $GT_n({\mathbb K})$. Then the matrix
$C^{-1}AC\in T_n({\A})$ and satisfies the Maurer-Cartan equation,
i.e. it is again the matrix of a formal connection.
\end{proposition}
\begin{proof}
$$
d(C^{-1}AC)-\bar C^{-1}\bar A\bar C \wedge C^{-1}AC=
C^{-1}\left(dA-\bar A\wedge A\right)C \in I_n({\mathcal A}).
$$
It follows also that the associated classes $[c(A)]$ and $[c(C^{-1}AC)]$ with $C$ from (\ref{GT_n}) are related 
$$
[c(C^{-1}AC)]=\frac{c_{n+1,n+1}[c(A)]}{c_{1,1}}.
$$
\end{proof}
\begin{example}
Let $A \in T_n({\A})$ be the matrix of a formal connection
(defining system) for a Massey product
$\langle \alpha_1,\dots, \alpha_n\rangle$.
Then a matrix $C^{-1}AC$ with
$$
C=\left(\begin{array}{ccccc}
  1      & 0 & \dots & 0   & 0   \\
  0      & x_1      &  \dots & 0  & 0   \\
 \dots   &        &       \dots &            & \dots     \\
  0      & 0            & \dots &  x_1{\dots}x_{n{-}1}& 0   \\
  0      & 0            & \dots & 0          & x_1{\dots}x_{n{-}1}x_n
  \end{array}\right)
$$
is a defining system for $\langle x_1\alpha_1,\dots, x_n\alpha_n\rangle=
x_1\dots x_n\langle \alpha_1,\dots, \alpha_n\rangle$.
\end{example}

\begin{definition}
\label{A_equiv}
Two matrices $A$ and $A'$ of formal connections
 are equivalent if there exists a
non-degerate scalar matrix $C \in GT_n({\mathbb K})$ such that
$
A'=C^{-1}AC.
$
\end{definition}
It is  obvious that we can consider only the subgroup of non-degenerate diagonal matrices 
instead of the whole $GT_n({\mathbb K})$ in the last definition.

Following the original Massey's paper~\cite{Mass},
some higher order cohomological operations
that we call now Massey products were introduced
in the 1960s in \cite{K} and \cite{JPM}. May briefly noticed in \cite{JPM} that there is a relation between the difinition of the Massey products and 
Maurer-Cartan equation. However, this analogy
was not developed till the Babenko-Taimanov paper~\cite{BaTa2}.

In the present article we deal only with Massey products of
non-trivial cohomology classes. In general situation it is more natural to consider so-called matric Massey products that were first introduced by May in
\cite{JPM} and developed in \cite{BaTa2}. 


\section{Massey products and Lie algebras representations}
\label{g-modules}
Consider the cochain complex with trivial coefficients ${\mathbb K}$ of an $n$-dimensional Lie algebra $\mathfrak{g}$
$$
\begin{CD}
\mathbb K @>{d_0{=}0}>> \mathfrak{g}^* @>{d_1}>> \Lambda^2 (\mathfrak{g}^*) @>{d_2}>>
\dots @>{d_{n-1}}>>\Lambda^{n} (\mathfrak{g}^*) @>>> 0.
\end{CD}
$$
where $d_1: \mathfrak{g}^* \rightarrow \Lambda^2 (\mathfrak{g}^*)$
is a dual mapping to the Lie bracket
$[ \, , ]: \Lambda^2 \mathfrak{g} \to \mathfrak{g}$.
The differential $d$ (the whole collection of $d_p$)
is the derivation of the exterior algebra $\Lambda^*(\mathfrak{g}^*)$
that continues $d_1$
$$
d(\rho \wedge \eta)=d\rho \wedge \eta+(-1)^{deg\rho} \rho \wedge d\eta,
\; \forall \rho, \eta \in \Lambda^{*} (\mathfrak{g}^*).
$$
It is easy to see that the condition $d^2=0$ is equivalent to the Jacobi identity of $\mathfrak{g}$.

Continue to use dual language and write with its help the definition of the representation of a Lie algebra by square  $(n+1,n+1)$-matrices.

\begin{proposition}
A $(n+1,n+1)$-matrix $A$ with entries from ${\mathfrak{g}}^*$ defines
a representation $\rho: \mathfrak{g} \to {\mathfrak gl}_n({\mathbb K})$
if and only if $A$ satisfies the strong Maurer-Cartan equation
$$
dA-\bar A \wedge A=0.
$$
\end{proposition}
\begin{proof}
$
(dA-\bar A \wedge A)(x,y)=A([x,y])-\left[ A(x), A(y)\right],\;\forall x,y \in
\mathfrak{g}.
$
\end{proof}

In the Section \ref{Massey_S}
the involution of a graded $\A$ was defined as $\bar a=(-1)^{k{+}1} a, a \in \A^k$.
Thus, for a matrix $A$ with entries in ${\mathfrak{g}}^*$
we have $\bar A =A$. One has to remark that $\bar a$ differs
by the sign from the definition of the involution $\bar a$ in \cite{K},
however, in \cite{JPM2} there is the same sign rule.

From now on, we will consider only representations in upper triangular matrices.
So, we denote by $T_n({\mathbb K}) \subset {\mathfrak gl}_n({\mathbb K})$ the Lie subalgebra of upper triangular $(n+1,n+1)$-matrices and consider representations
$\rho: \mathfrak{g} \to T_n({\mathbb K})$ of
a Lie algebra $\mathfrak{g}$ for some fixed value $n$. 

Consider $n=1$ and a linear map
$$
\rho: x \in \mathfrak{g} \to
A(x)=\left(\begin{array}{cc}
0 & a(x) \\
0& 0
\end{array}\right).
$$
It is evident that
$\rho$ is a Lie algebra homomorphism if and only if
the linear form $a \in {\mathfrak{g}}^*$ is closed.
Or equivalently,
$A$
satisfies the strong Maurer-Cartan equation
$dA-\bar A \wedge A =0$.

The Lie algebra $T_n({\mathbb K})$ has a
one-dimensional center $I_n({\mathbb K})$ spanned by the matrix
$$
\left(\begin{array}{cccc}
0 & \dots & 0 & 1 \\
0 & \dots & 0 & 0 \\
& \dots& & \\
0 & \dots & 0 & 0
\end{array}\right).
$$
One can consider an one-dimensional central extension
$$
\begin{CD}
0 @>>> \mathbb K \cong I_n({\mathbb K}) @>>>
T_n(\mathbb K) @>{\pi}>> \tilde T_n(\mathbb K) @>>> 0.
\end{CD}
$$

\begin{proposition}[\cite{FeFuRe}, \cite{Dw}]
Fixing a Lie algebra homomorphism $\tilde \varphi : \mathfrak{g} \to
\tilde T_n({\mathbb K})$ is equivalent to fixing a defining system
$A$ with elements from $\mathfrak{g}^*=\Lambda^1(\mathfrak{g})$.
The related cocycle $c(A)$ is cohomologious to zero
if and only if $\tilde \varphi$ can be lifted to a
homomorphism $\varphi : \mathfrak{g} \to  T_n({\mathbb K})$,
$\tilde \varphi = \pi \varphi$.
\end{proposition}

There is a standard definition.
\begin{definition}
Two representations
$\varphi : \mathfrak{g} \to T_n({\mathbb K})$ and
$\varphi' : \mathfrak{g} \to T_n({\mathbb K})$ are called 
equivalent if there  exists
a  non-degenerate matrix $C \in GL(n{+}1,{\mathbb K})$ such that
$$
\varphi'(g)=C^{-1}\varphi(g)C, \; \forall g \in \mathfrak{g}.
$$
\end{definition}
It is evident that this definition
is equivalent to the Definition \ref{A_equiv} for 
Massey products 
of $1$-cohomology classes $\omega_1,\dots, \omega_n$. Moreover, these linear forms are on the second (top) main diagonal of the matrix $A$
of the formal connection.

\begin{proposition}
Let a Massey product $\langle \omega_1,\omega_2,{\dots},\omega_n \rangle$
be defined and trivial in $H^2(\mathfrak{g})$ for some $1$-cohomology classes
$\omega_i \in H^1(\mathfrak{g})$ of a Lie algebra $\mathfrak{g}$. Then
$\langle x_1\omega_1,x_2\omega_2,{\dots},x_n\omega_n \rangle$
is also defined and trivial for any choice of non-zero constants
$x_1,x_2,{\dots},x_n$.
\end{proposition}


\section{$k$-step Massey products in Lie algebra cohomology}
\label{k-step}
This paragraph is written largely under the influence of an article by Ido Efrat  \cite{efrat}. His main observation was
the following theorem.
\begin{theorem}[Efrat~\cite{efrat}]
\label{ThEfr}
Suppose that for every set  $\omega_1,\dots,\omega_n$  of  $1$-cocycles from $H^1({\mathfrak g})$ there is a matrix $A$ of formal connection. Then the Massey product 
$$
\langle \cdot,  \dots,\cdot \rangle: \underbrace{H^1(\mathfrak g)\times \dots \times H^1(\mathfrak g)}_{n \rm\ times} \to H^2({\mathfrak g})
$$ 
is a well-defined single-valued map. 
\end{theorem}
However, in this section we want to discuss other ideas of \cite{efrat} in a revised form, in particular, to define higher order Massey products $\langle a_1, a_2, \dots, a_n\rangle$ as successive obstructions $\langle a_1, a_2, \dots, a_n\rangle_k, k=1, 2, \dots, n-1,$ to the existence of a formal connection $A$ for a given set of one-dimensional cocycles $a_1, a_2, \dots, a_k$ from the differential algebra ${\mathcal A}$.

Consider the descending series  $\{ T_n^{n-k}({\mathcal A}), k=0,\dots, n-1\}$ of two-sided ideals
$$
T_n^{n-1}({\mathcal A})=T_n({\mathcal A}) \subset T_n^{n-2}({\mathcal A})
\subset \dots \subset T_n^{n-k}({\mathcal A}) \subset  \dots
\subset T_n^0({\mathcal A})=I_n({\mathcal A}) \subset \{0\},
$$
where $T_n^{n-k}({\mathcal A})$ denotes the subspace  of
upper triangular $(n{+}1,n{+}1)$-matrices $B$ with entries from  ${\mathcal A}$
of the following form
$$
B=\left(\begin{array}{ccccccc}
  0&0&\dots& 0&        b(1,k) & \dots  & b(1,n)   \\
  0&\ddots & \ddots& \ddots&0  &\ddots  & \vdots   \\
 \vdots   &\ddots&\ddots&  \ddots      &       \ddots &   \ddots         & b(n{-}k,n)     \\
    \vdots&& \ddots &   \ddots        &  \ddots& \ddots &   0 \\
      \vdots&&  &  \ddots         & \ddots &  \ddots&    \vdots\\
    \vdots    &&  &             &\ddots & \ddots&0    \\
  0     &\dots& \dots& \dots            & \dots & 0         & 0\\
  \end{array}\right), b(i,j) \in {\mathcal A}.
$$
Obviously for non-negative integers $k,m, k+m \le n,$ we have inclusions
\begin{equation}
\label{triangular_mult}
 T_n^{n-k}(\mathcal A) \cdot  T_n^{n-m}(\mathcal A) \subset  T_n^{n-k-m}(\mathcal A).
\end{equation}
\begin{definition}
Consider a non-negative integer $k, 0 \le k \le n-1$. A matrix $A_k \in T_n(\mathcal A)$ is called the matrix of
a $k$-step formal connection if it satisfies the $k$-step  Maurer-Cartan equation
\begin{equation}
\mu(A_k)=d\,A_k-\bar{A_k}\cdot A_k \in T_n^{n-k}(\mathcal A).
\label{k-star}
\end{equation}
\end{definition}
\begin{proposition}
Let $A_k$ be the matrix of a $k$-step formal coonection. Then  
$$
d\mu(A_k) \in T_n^{n-k-1}(\mathcal A).
$$
\end{proposition}
\begin{proof}
It follows from (\ref{triangular_mult}). Indeed  the generalized Bianci identity
(\ref{MCar}) means
$$
\begin{array}{c}
d\,\mu(A_k)=-d\,(\bar{A_k}\cdot A_k)
=\overline{\mu(A_k)}\cdot A_k+A_k\cdot\mu(A_k) \in  T_n^{n-k}(\mathcal A) \cdot  T_n^{n-1}(\mathcal A).
\end{array}
$$
\end{proof}
\begin{corollary}
Let $A_k$ be the matrix of a $k$-step formal connection, $0 \le k \le n{-}1$.
Then all elements on the $k$-th parallel main diagonal of the matrix $\mu(A_k)=(\tau(i,j))$ are closed
$$
d\tau(1,k)=0, \; d\tau(2,k{+}1)=0,\; \dots,\;  d\tau(n{-}k,n)=0.
$$
\end{corollary}
From now on, we will consider matrices of formal connections with elements $\omega$ from the dual space  ${\mathfrak g}^*$ to the Lie algebra ${\mathfrak g}$. It means that they are all $1$-forms and $\bar \omega=\omega$ for every $\omega \in {\mathfrak g}$.
\begin{definition}
Let $A_k$ be the matrix of a $k$-step formal connection, $0 \le k \le n{-}1$.
The set  of $n-k$
two-dimensional cocycles
$$
c(A_k)=\left(\sum_{r=1}^{n-1}\bar{a}(1,r)\wedge a(r{+}1,k), \dots, \sum_{r=1}^{n-1}\bar{a}(n{-}k,r)\wedge a(r{+}1,n) \right).
$$
is called the related set $c(A_k)$ of cocycles of the formal $k$-step connection  $A_k$.
\end{definition}

\begin{definition}
\label{def-k-step}
The $k$-step  $n$-fold product $\langle a_1,\dots,a_n\rangle_k$ is defined if
there exists at least one $k$-step
formal connection $A_k$ for it with
entries $a_1,\dots,a_n$ at the second diagonal. If it is defined, then
the value $\langle a_1,\dots,a_n\rangle_k$ is the set of all $(n-k)$-tuples of cohomology
classes $\left(\alpha_1,\dots, \alpha_{k+1}\right) \in \underbrace{H^{2}(\A)\times\dots H^{2}(\A)}_{n-k}$ for which
there exists a $k$-step formal connection $A_k$
such that the sequence $c(A_k)$ (or equivalently $\left(-\tau_1,\dots, -\tau_{n-k}\right)$) represents $\left(\alpha_1,\dots, \alpha_{n-k}\right)$.
\end{definition}

$$
\mu(A_k)=\left(\begin{array}{ccccccc}
  0&\dots& 0&        \tau_1 &\;\;*& \dots  & *   \\
  0 & & &0&\tau_2  &\ddots  & \vdots   \\
  \vdots & & & & \ddots &\ddots & *   \\
 \vdots&&        &        &         &0  & \tau_{n-k}     \\
 \vdots&   & &           & &  &   0 \\
   0 &  &  &  \dots         & \dots &  \dots&    \vdots\\
  \end{array}\right), \tau_i \in {\mathfrak  g}^*, d\tau_i=0, i=1,\dots, n{-}k.
$$
Using the standard arguments from \cite{K, BaTa2} one can prove that 
the product  $\langle a_1,\dots,a_n\rangle_k$ depends only on the
cohomology classes of the elements $a_1,\dots,a_n$ as it holds in the classic case.
\begin{example}
\label{1-step}
The $1$-step  product $\langle a_1,\dots,a_n\rangle_1$ is defined, single-valued and equal to the following $(n-1)$-tuple of double products
$$
\langle a_1,\dots, a_n \rangle=\left( a_1 \wedge a_2,  a_2 \wedge a_3, \dots,  a_{n-2} \wedge a_{n-1}, a_{n-1} \wedge a_n \right) \in 
\underbrace{H^{2}({\mathfrak g})\times\dots H^{2}({\mathfrak g})}_{n-1}.
$$
\end{example}
\begin{remark}
If we do not pay attention to the subscript $k$, the $n$-fold  Massey product $\langle a_1,\dots,a_n\rangle$ is now always defined, although we a priori do not know in which space its value will be located.
\end{remark}
By analogy with the classical case, we give the following definition
\begin{definition}
Let a $k$-step  $n$-fold Massey product
$\langle {\alpha}_1,\dots,{\alpha}_n\rangle_k$ in the Lie algebra cohomology
be defined. It is called trivial if it contains the trivial cohomology class of the 
vector space $\underbrace{H^{2}({\mathfrak g})\times \dots \times H^{2}({\mathfrak g})}_{n-k}$
$$
(0,\dots,0) \in\langle {\alpha}_1,\dots,{\alpha}_n\rangle_k.
$$
\end{definition}

\begin{proposition}
\label{triviality}
Let a $k$-step Massey product
$\langle {\alpha}_1,\dots,{\alpha}_n\rangle_k$
be defined. Then all Massey products
$\langle {\alpha}_l,\dots,{\alpha}_q\rangle_s, 1\le l < q \le n, q-l<n-1, s \le k,$
are defined and trivial.
\end{proposition}

\begin{proposition}
Let a $k$-step Massey product
$\langle {\alpha}_1,\dots,{\alpha}_n\rangle_k$
be defined and trivial. Then the $(k+1)$-step Massey product
$\langle {\alpha}_1,\dots,{\alpha}_n\rangle_{k+1}$ is defined.
\end{proposition}
Returning to Example \ref{1-step}, we see that a two-step formal connection $A_2$ for the classes $a_1,\dots,a_n,$
 exists if and only if all the cocycle products $a_i \wedge a_{i+1}, i=1,\dots,n{-}1,$ are trivial in the cohomology $H^*({\mathcal A})$. Moreover, if we find $1$-forms
$a(i,i{+}1)$ solving the equations
$$
a(i,i) \wedge a(i{+}1,i{+}1)=da(i,i{+}1), i=1,\dots,n{-}1,
$$
we can explicitly write down the formal connection matrix  $A_2$ with their help and (\ref{def_syst}). New elements $a(i,i{+}1)$ make up the second diagonal of the matrix $A_2$.

\begin{corollary}
Let a standard $n$-fold Massey product
$\langle {\alpha}_1,\dots,{\alpha}_n\rangle$
be defined. Then all $k$-step Massey products
$\langle {\alpha}_1,\dots,{\alpha}_n\rangle_{k}$ with $k < n$ are defined and trivial.
\end{corollary}
\begin{proof}
It means that the $(n-1)$-step Massey product $\langle {\alpha}_1,\dots,{\alpha}_n\rangle_{n-1}$ is defined. Hence the statement follows from
Proposition \ref{triviality}.
\end{proof}

To conclude this section, we want to note that the technique of $k$-step Massey products constructed in it is not something completely unknown. Anyone who searched for (constructed) the defining systems of Massey products came across this recursive procedure at one level or another. But  it seems however useful to formalize it in this specific way thinking on the applications to representations theory of nilpotent Lie algebras.


\section{Non-trivial Massey products in Lie algebra cohomology}
\label{nontrivSection}
Buchstaber and Shokurov discovered \cite{BuSho} that the tensor product 
$S \otimes {\mathbb R}$ of the
Landweber-Novikov algebra $S$ (the complex cobordism theory) 
by real numbers $S \otimes {\mathbb R}$ is isomorphic to the
universal enveloping algebra $U(W^+)$ of the Lie algebra $W^+$ of
polynomial vector fields on the real line ${\mathbb R}^1$ with
vanishing non-positive Fourier coefficients. $W^+$ is a maximal
pro-nilpotent subalgebra of the Witt algebra.

The Witt algebra $W$ is
spanned by differential operators on the real line ${\mathbb R}^1$
with a fixed coordinate $x$
$$
e_i=x^{i+1}\frac{d}{dx}, \; i \in {\mathbb Z}, \quad
[e_i,e_j]= (j-i)e_{i{+}j}, \; \forall \;i,j \in {\mathbb Z}.
$$

\begin{theorem}[Goncharova~\cite{G}]
The Betti numbers ${\rm dim} H^q(W^+)=2$, for every
$q \ge 1$, more precisely
$$ {\rm dim} H_k^q(W^+)=
\left\{\begin{array}{r}
   1, \hspace{0.6em}{\rm if}~k=\frac{3q^2 \pm q}{2}, \\
   0,\hspace{1.76em}{\rm otherwise.}\\
   \end{array} \right . \hspace{3.3em} $$
\end{theorem}
We will denote  by $g^q_{\pm}$ a basis in the
spaces $H^q_{\frac{3q^2\pm q}{2}}(W^+)$. The numbers $\frac{3q^2
\pm q}{2}=e_{\pm}(q)$ are so called Euler pentagonal numbers. It is easy to verify that the sum 
$$
e_{\pm}(q)+e_{\pm}(p) \neq e_{\pm}(p+q), p, q \in {\mathbb N}.
$$ 
Hence the
cohomology algebra $H^*(W^+)$ has a trivial multiplication.
Buchstaber conjectured that the algebra $H^*(W^+)$ is generated
with respect to some non-trivial Massey products by its  first
cohomology $H^1(W^+)$.

Feigin, Fuchs and Retakh \cite {FeFuRe} represented the basic
homogeneous cohomology classes from $H^*(W^+)$ as  Massey products. 

\begin{theorem}[Feigin, Fuchs, Retakh~\cite{FeFuRe}]
\label{FeigFu}
For any $q \ge 2$ we have inclusions
$$
g^q_- \in \langle g^{q{-}1}_+,\underbrace{e^1,\dots,e^1}_{2q-1}
\rangle,
\quad
g^q_+ \in \langle g^{q{-}1}_+,\underbrace{e^1,\dots,e^1}_{3q-1}
\rangle.
$$
\end{theorem}

Feigin, Fuchs and Retakh proposed to consider the defining system that is equivalent to the following matrix of formal connection
$$
\label{FeFuReCon} 
A=\left(\begin{array}{ccccccc}
  0      &g^{k{-}1}_+ &   \Omega_1 &\Omega_2 & \dots   &  \Omega_{n{-}1} &*  \\
    0& 0 &  e^1 & \alpha e^2 & 0  & \dots  &0 \\
  0&0  & 0 & \ddots & \ddots      & \ddots &\vdots   \\
       0& 0      & 0      & \ddots & \ddots          & \alpha e^2 &0   \\
  \vdots      & 0      &  \ddots    & \ddots &  \ddots & e^1 & \alpha e^2 \\
 0 & \vdots      & \ddots      & \ddots      & \ddots & 0          & e^1  \\
  0 & 0      & 0      & 0      & \dots & 0  &  0
  \end{array}\right),
$$
with homogeneous forms $\Omega_i \in
\Lambda^{k+i-1}_{\frac{1}{2}(3(k{-}1)^2{+}(k{-}1))+i}(W^+)$ and
parameter $\alpha \in {\mathbb K}$.  The corresponding cocycle $c(A)
\in
\Lambda^{k+1}_{\frac{1}{2}(3(k{-}1)^2{+}(k{-}1)){+}n{-}1}(W^+)$
can be non-trivial only if $n=2k$ or $n=3k$ that corresponds to
$H^k_{\frac{1}{2}(3k^2\pm k)}(W^+)$.  Feigin, Fuchs and Retakh have shown that the triviality of the cocycle $c(A)$ is
equivalent to the triviality of the differential $d_n$ of some
spectral sequence $E_r^{p,q}$ converging to the cohomology $H^*(W^+,V)$ with coefficients in the graded $W^+$-module $V$. The infinite dimensional module $V$ depends on the parameter $\alpha$ and can be defined by its basis $f_1,f_2,\dots, f_n,\dots$ and relations
$$
e_1f_j=f_{j+1}, \; e_2f_j=\alpha f_{j+2}, j \ge 1; \quad e_kf_j=0, k,j \in {\mathbb N}.
$$

Feigin, Fuchs and Retakh established that

1) the differential $d_{2k}$ is trivial if and only if $\alpha \in
\{\frac{1}{6},\frac{1}{24},\dots,\frac{1}{6(k{-}1)^2}\}$;

2) the differential $d_{3k}$ is defined and trivial if and only if:

a) $\alpha \in
\{\frac{1}{6},\frac{1}{54},\dots,\frac{1}{6(k{-}3)^2},\frac{1}{6(k{-}1)^2}\}$
in the case of even $k$;

b) $\alpha \in
\{\frac{1}{24},\frac{1}{96},\dots,\frac{1}{6(k{-}3)^2},\frac{1}{6(k{-}1)^2}\}$
if $k$ is odd.
\begin{corollary}
All Massey products from the Theorem \ref{FeigFu} are trivial.
\end{corollary}
The main
technical problem in the proof of Theorem \ref{FeigFu} is  to find explicit
formulas for the entries $\Omega_i, i=1,\dots, n-1$.
One can verify directly that the cocycles $g^2_-=e^2{\wedge}e^3$ and
$g^2_+=e^2{\wedge}e^5-e^3{\wedge}^4$ span the homogeneous
subspaces $H^2_5(W^+)$ and $H^2_7(W^+)$ respectively. But  explicit formulas for all
Goncharova's cocycles $g^k_{\pm}$ in terms of exterior forms from
$\Lambda^*(W^+)$ are still unknown. Fuchs, Feigin and Retakh proposed \cite{FeFuRe}
an elegant  way how to establish
non-triviality of differentials for some values of the
parameter $\alpha$ of the spectral sequence $E^{p,q}_r$ that converge to the cohomology $H^*(W^+,V)$.

Artel'nykh \cite{Artel} represented a
part of basic cocycles in $H^*(W^+)$ by means of non-trivial
Massey products, but his very brief article does not contain any sketch of the proof.

\begin{theorem}[Millionshchikov~\cite{Mill2}]
\label{main_Buch_conj}
The cohomology $H^*(W^+)$ is generated by two
elements $e^1, e^2 \in H^1(W^+)$ by means of two series of
non-trivial Massey products. More precisely the recurrent
procedure is organized as follows

1) elements $e^1$ and $e^2$ span $H^1(W+)$;

2) the triple Massey product $\langle e^1,e^2,e^2 \rangle$
 is single-valued and determines non-trivial cohomology
class $g_-^{2}{=}\langle e^1,e^2,e^2 \rangle \in H^{2}_{5}(W^+)$;

3) the $5$-fold product $\langle e^1,e^2,e^1,e^1,e^2 \rangle$ is
non-trivial and it is an affine line $\{g_+^{2}+t g_-^{2}, t \in
{\mathbb K}\}$ on the plane  $H^{2}(W^+)$, where $g_+^{2}$ denotes
some generator from $H^{2}_{7}(W^+)$. Denote by $\tilde g_+^{2}$  an arbitrary
element in $\langle e^1,e^2,e^1,e^1,e^2 \rangle$.

Let us suppose that we have already constructed some basis $g_-^k,
\tilde g_+^k$ of $H^k(W^+), k {\ge} 2,$ such that the cohomology
class $g_-^k$ spans the  subspace
$H^k_{\frac{3k^2-k}{2}}(W^+)$. Then

4) the $(2k+1)$-fold Massey product
$$ \langle
\underbrace{e^1,\dots,e^1}_m,e^2,
\underbrace{e^1,\dots,e^1}_n,\tilde g_+^k\rangle=g_-^{k+1}, \; m+n=2k-1,$$
 is single-valued and spans the subspace
$H^{k{+}1}_{\frac{1}{2}(3(k{+}1)^2{-}(k{+}1))}(W^+)$.

5) the $(3k+2)$-fold product
$$ \langle
\underbrace{e^1,\dots,e^1}_{k},e^2,\underbrace{e^1,\dots,e^1}_{2k},\tilde
g^k_+ \rangle$$
 is non-trivial and it is an affine line on the two-dimensional plane $H^{k+1}(W^+)$ parallel to the
 one-dimensional subspace
 $H^{k{+}1}_{\frac{1}{2}(3(k{+}1)^2{-}(k{+}1))}(W^+)$.
One can take an arbitrary element  in $\langle
\underbrace{e^1,\dots,e^1}_{k},e^2,\underbrace{e^1,\dots,e^1}_{2k},\tilde
g^k_+ \rangle$ as the second basic element $\tilde g_+^{k+1}$ of the subspace $H^{k+1}(W^+)$.
\end{theorem}

The proof of Theorem \ref{main_Buch_conj}  is based on the technique developed in \cite{FeFuRe}, to which a set of fundamental additions was proposed \cite{Mill2}. The first was the construction of the special graded thread $W^+$-module $V_{gr}$, which is uniquely determined by its main properites. The non-triviality of the corresponding Massey products largely follows from such a rigidity. To calculate the differentials of the spectral sequence associated with the constructed Massey products, explicit formulas for the special vectors of Verma modules over the Virasoro algebra were applied. Finally the nontriviality of the corresponding differentials $d_k$ was established by explicit calculus of the cohomology $H^*(W^+, V_{gr})$ using the so-called Feigin-Fuchs-Rocha-Caridi-Wallach resolution \cite{Mill2}.

As an example, we will reproduce the simplest part of the proof \cite{Mill2} related to the second cohomology  $H^*(W^+)$, but which nevertheless well illustrates the basics of the proof technique.

For an arbitrary formal connection $A$ that corresponds to the
product $\langle e^2,e^2,e^1\rangle$ the cocycle
$c(A)=-e^2 \wedge e^3+\alpha d(e^3)$ for some
scalar $\alpha$. Hence the single-valued triple product 
$\langle e^1,e^2,e^2\rangle=-[e^2 \wedge
e^3]\ne 0$ spans the subspace $H^2_5(W^+)$.

Consider  $\langle e^1,e^2,-e^1,-2e^1,-e^2 \rangle$
instead of $\langle e^1,e^2,e^1,e^1,e^2 \rangle$ and a formal connection $A$
$$ A=
\left(\begin{array}{cccccc}
0 & e^1 & e^3  &e^4  &e^5 &*    \\
  0 & 0 & e^2  &e^3  &e^4 &0    \\
   0   & 0   & 0       &-e^1  &e^2&-e^4{-}te^2     \\
     0 & 0 &0  & 0 & -2e^1 &2e^3       \\
      0&  0 & 0 & 0 & 0 &-e^2 \\
      0&   0 & 0 & 0&0  &0          \\
\end{array}\right).
$$
The corresponding cocycle will be $c(A)=(e^2{\wedge}
e^5{-}3e^3{\wedge} e^4)+te^2{\wedge} e^3$. On the another hand for
an arbitrary defining system $A'$ the corresponding cocycle will
have the form $c(A')=(e^2{\wedge} e^5{-}3e^3{\wedge}
e^4){+}{\dots}$, where dots stand for the summands with the second
grading strictly less than $7$.

\begin{example}
The Lie algebra $\mathfrak{m}_0$ is defined
by its infinite basis $e_1, e_2, \dots, e_n, \dots $
with commutator relations:
$$ \label{m_0}
[e_1,e_i]=e_{i+1}, \; \forall \; i \ge 2; [e_i,e_j]=0, i,j \neq 1.
$$
\end{example}

Let $\mathfrak{g}$ be a pro-nilpotent or ${\mathbb N}$-graded Lie algebra.
The ideals $\mathfrak{g}^{k}$ of the descending central sequence define
a decreasing filtration  of the Lie algebra 
$\mathfrak{g}$
$$
\mathfrak{g}=\mathfrak{g}^{1} \supset \mathfrak{g}^{2} \supset
\dots \supset \mathfrak{g}^{k} =[\mathfrak{g}^{1}, \mathfrak{g}^{k-1}] \supset \dots; \;
[\mathfrak{g}^{k},\mathfrak{g}^{l}] \subset \mathfrak{g}^{k+l}, k,l \in {\mathbb N}.
$$
Consider the associated graded Lie algebra
$$
{\rm gr}\: \mathfrak{g}=\bigoplus_{k=1}^{+\infty} ({\rm gr}\: \mathfrak{g})_k, \;\;
({\rm gr} \: \mathfrak{g})_k=\mathfrak{g}^{k}/\mathfrak{g}^{k{+}1}, k \in {\mathbb N}.
$$

\begin{proposition}
We have the following isomorphisms:
$$ 
{\rm gr} \:W_+ \cong {\rm gr}\: \mathfrak{m}_0 \cong
\mathfrak{m}_0.
$$
\end{proposition}
The corresponding natural grading of ${\mathfrak m}_0$ is defined by
$$
({\rm gr}\: \mathfrak{m}_0)_1=Span(e_1,e_2),\;
({\rm gr} \:\mathfrak{m}_0)_i=Span(e_{i{+}1}), \;i\ge 2.
$$

Let $\mathfrak{g}=\oplus_{\alpha}\mathfrak{g}_{\alpha}$ be a
${\mathbb N}$-graded (pro-nilpotent) Lie algebra
and $V$ is a finite-dimensional nilpotent ${\mathfrak g}$-module. 
There is a decreasing filtration of the ${\mathfrak g}$-module module $V$
$$
V^1=V \subset V^2={\mathfrak g}V \subset \dots V^k={\mathfrak g}V^{k-1} \subset 
\dots
$$
One can define
the associated graded module $gr\: V$ over the associated graded Lie algebra $gr\: {\mathfrak g}$
$$
gr\: V=\oplus_{i=1}^{+\infty}(gr\: V)_i, \: gr\:V_i=V^i/V^{i+1}, (gr\: {\mathfrak g})_i (gr\: V)_j \subset 
(gr\: V)_{i+j}, \: i, j \in {\mathbb N}.
$$

Thus, we came to the problem of description of Massey products
in the cohomology $H^*(\mathfrak{m}_0)$. As we saw earlier, it is useful to describe trivial Massey products
$\langle \omega_1,\dots,\omega_n\rangle$ of $1$-cohomology classes
$\omega_1,\dots,\omega_n$. The purpose of this interest is
to consider
Massey products of the form $\langle \omega_1,\dots,\omega_n,\Omega\rangle$,
where $\Omega$ is a cocycle from $H^p(\mathfrak{g})$ for $p>1$.

Consider $p=2$.
It was found out in \cite{FialMill} that  $H^2(\mathfrak{m}_0)$ is spanned
by the cohomology classes of the following set of cocycles \cite{FialMill}
\begin{equation}
\begin{array}{c}
\label{2-hom}
\omega(e^k{\wedge} e^{k{+}1})=
\sum_{l=0}^{k-2} ({-}1)^l e^{k-l}{\wedge} e^{k{+}1{+}l}, k \ge 2.
\end{array}
\end{equation}

All of the cocycles (\ref{2-hom}) can be represented as Massey products.
Namely let us consider the following matrix of a formal connection
$$
A=\left(\begin{array}{ccccccc}
  0      & e^2 & -e^3 & \dots  & (-1)^ke^k & (-1)^{k{+}1}e^{k{+}1}   & 0  \\
  0      & 0      & e^1 & 0& \dots & 0   & e^{k{+}1}   \\
  0      & 0      & 0 & e^1 &  0 & \dots   & e^k   \\
  \dots & \dots &   \dots     & \dots  &\dots  & \dots & \dots  \\
  0  &0    &  0&   \dots  &  0     &  e^1 & e^3 \\
  0   &0   & 0      & 0      & \dots & 0          & e^2   \\
  0    &0  & 0      & 0      &  0& \dots          & 0
  \end{array}\right).
$$
For the related cocycle $c(A)$ we have
$$
c(A)=\sum_{l=2}^{k{+}1}(-1)^l e^l \wedge e^{k+3-l}=
2\omega(e^k{\wedge}e^{k{+}1}).
$$
So we proved that
$$
2\omega(e^k{\wedge}e^{k{+}1}) \in
\langle e^2, \underbrace{e^1, \dots, e^1}_{2k-3}, e^2 \rangle,
\quad k \ge 2.
$$

The space $H^1(\mathfrak{m}_0)$ is spanned by $e^1$ and $e^2$
and therefore an arbitrary $n$-fold Massey product of
elements from $H^1(\mathfrak{m}_0)$ has a form
$$
\underbrace{\langle \alpha_1 e^1 +\beta_1 e^2, \alpha_2 e^1 +\beta_2 e^2,
\dots, \alpha_n e^1 +\beta_n e^2 \rangle}_n.
$$
It follows from $e^1 \wedge e^2=de^3$ that
a triple product
$$
\langle \omega_1, \omega_2, \omega_3 \rangle=
\langle \alpha_1 e^1 +\beta_1 e^2, \alpha_2 e^1 +\beta_2 e^2,
\alpha_3 e^1 +\beta_3 e^2 \rangle
$$
is defined for  all values
$\alpha_i, \beta_i \in {\mathbb K}, i=1,2,3.$
$$
A=\left(\begin{array}{cccccc}
0 & \omega_1 & \gamma_1e^3& 0\\
0 & 0 & \omega_2& \gamma_2e^3\\
0 & 0 & 0& \omega_3\\
0 & 0 & 0& 0 \\
\end{array}\right), \; \gamma_1=\alpha_1\beta_2-\alpha_2\beta_1,
\gamma_2=\alpha_2\beta_3-\alpha_3\beta_2,
$$
The related cocycle $c(A)=\gamma_2\omega_1{\wedge}e^3-
\gamma_1\omega_3{\wedge}e^3$ is trivial if and only if
\begin{equation}
\label{triple_M}
\beta_1(\alpha_2\beta_3-\alpha_3\beta_2)-
\beta_3(\alpha_1\beta_2-\alpha_2\beta_1)=0.
\end{equation}

Let us consider the operator $D_1=ad^*e_1: \Lambda^*(e_2, e_3, \dots)
\to \Lambda^*(e_2, e_3, \dots)$ acting on the chain subcomplex  $\Lambda^*(e_2, e_3, \dots) \subset \Lambda^*(e_1, e_2, e_3, \dots)$ of ${\mathfrak m}_0$.
\begin{equation}
\begin{array}{c}
D_1(e^2)=0, \; D_1(e^i)= e^{i-1}, \; \forall i\ge 3,\\
D_1(\xi\wedge \eta)=D_1(\xi)\wedge \eta +\xi\wedge D_1(\eta),
\; \: \forall \xi, \eta \in \Lambda^*(e_2, e_3, \dots).
\end{array}
\end{equation}

The operator $D_1$ has the right inverse operator
 $D_{-1}: \Lambda^*(e^2,e^3,\dots) \to
\Lambda^*(e^2,e^3,\dots)$,
defined by the formulas
\begin{equation}
\begin{split}
\label{D_{-1}}
D_{-1}e^i=e^{i+1}, \;
D_{-1}(\xi{\wedge} e^i)=
\sum_{l\ge 0}(-1)^l D_{1}^{l}(\xi){\wedge} e^{i+1+l},
\end{split}
\end{equation}
where $i\ge 2$ and $\xi$ stands for an arbitrary form in
$\Lambda^*(e^2,\dots,e^{i-1})$.
One can verify that 
$D_1D_{-1}=Id$ on $\Lambda^*(e^2,e^3,\dots)$.

The sum in the definition (\ref{D_{-1}}) of $D_{-1}$ is always finite
because $D_1^l$ strictly decreases the second grading by $l$.
For instance,
$D_{-1}(e^i\wedge e^k)=\sum_{l=0}^{i-2} ({-}1)^l e^{i-l}{\wedge} e^{k+l+1}$.

There is an explicit  formula for
$D_{-1}(e^{i_1} {\wedge} \dots {\wedge} e^{i_q}{\wedge} e^{i_q})=D_{-1}(0)$
$$
\omega(e^{i_1}{\wedge}\dots {\wedge} e^{i_q} {\wedge}
e^{i_q{+}1})=
\sum\limits_{l\ge 0}(-1)^l D_1^l(e^{i_1}\wedge \dots
\wedge e^{i_q})\wedge e^{i_q+1+l},$$
this sum is also always finite and determines
a homogeneous closed $(q{+}1)$-form
of the second grading $i_1{+}{\dots}{+}i_{q{-}1}{+}2i_{q}{+}1$.

\begin{theorem}[Fialowski, Millionshchikov~\cite{FialMill}]
\label{main_H_m_0}
The bigraded cohomology algebra
$H^*(\mathfrak{m}_0)=\oplus_{k,q} H^q_k(\mathfrak{m}_0)$
is spanned by the cohomology classes of the following homogeneous cocycles:
\begin{equation}
\label{sumHm_0}
\begin{split}
e^1, \: e^2, \:
\omega(e^{i_1}{\wedge}\dots
{\wedge} e^{i_q}{\wedge} e^{i_q{+}1})=
\sum\limits_{l\ge 0}({-}1)^l (ad^* e_1)^l
(e^{i_1}{\wedge} e^{i_2}{\wedge} \dots
{\wedge} e^{i_q}){\wedge} e^{i_q{+}1{+}l}, \\
\end{split}
\end{equation}
where
$q \ge 1, \; 2\le i_1 {<}i_2{<}{\dots} {<}i_q$.

The multiplicative structure is defined by the rules
\begin{multline}
[e^1] {\wedge} \omega(\xi{\wedge}e^i{\wedge}e^{i+1})=0,\;
[e^2] {\wedge} \omega(\xi{\wedge}e^i{\wedge}e^{i+1})
=\omega(e^2{\wedge}\xi{\wedge}e^i{\wedge}e^{i+1}),\\
\omega(\xi{\wedge}e^i{\wedge}e^{i+1})
{\wedge} \omega(\eta{\wedge}e^j{\wedge}e^{j+1})
=\sum_{l=0}^{j-i-2}(-1)^l\omega((ad^*e_1)^l(\xi{\wedge}e^i){\wedge}e^{i+1+l}
{\wedge}\eta{\wedge}e^j{\wedge}e^{j+1})+\\
+(-1)^{i{-}j{+}deg\eta}\sum_{s \ge 1}
\omega((ad^*e_1)^{i{-}j{-}1{+}s}(\xi{\wedge}e^i){\wedge}
(ad^*e_1)^s(\eta{\wedge}e^j){\wedge}e^{j{+}s}
{\wedge}e^{j{+}s{+}1}),
\end{multline}
where $i < j$, $\xi$ and $\eta$ are arbitrary homogeneous forms in
$\Lambda^*(e^2,\dots,e^{i{-}1})$ and $\Lambda^*(e^2,\dots,e^{j{-}1})$,
respectively.
\end{theorem}
It can be seen that the cohomology of the two algebras ${\mathfrak m}_0$ and $W^+$ are very different from each other. The cohomology $H^p({\mathfrak m}_0)$ is infinite-dimensional for $p \ge 2$, while ring multiplication in cohomology $H^*({\mathfrak m}_0)$ is not trivial, although there is a sufficient number of trivial products, in particular, the product mapping $H^1({\mathfrak m}_0) \wedge H^1({\mathfrak m}_0) \to H^2({\mathfrak m}_0)$ vanishes.

We recall that if an $n$-fold Massey product
$\langle \omega_1, \omega_2, {\dots},\omega_n \rangle$
is defined than all $(p+1)$-fold Massey products
$\langle \omega_i, \omega_{i+1}, \dots,\omega_{i+p} \rangle$  for
$1\le i \le n-1, 1\le p \le n-2, i+p \le n$
are defined and trivial.

The following theorem shows that cohomology $H^*(\mathfrak{m}_0)$, as well as $H^*(W^+)$ is generated by non-trivial Massey products, if we include in their number ordinary wedge products, as double Massey products.

\begin{theorem}[Millionshchikov~\cite{Mill3}]
\label{2main_th}
The cohomology algebra  $H^*(\mathfrak{m}_0)$  is generated
with respect to the non-trivial Massey products
by $H^1(\mathfrak{m}_0)$, namely
\begin{equation}
\begin{split}
\omega(e^{2}{\wedge}e^{i_2}{\wedge}{\dots}
{\wedge} e^{i_q}{\wedge} e^{i_q{+}1})=
e^{2}{\wedge}\omega (e^{i_2}{\wedge}{\dots}
{\wedge} e^{i_q}{\wedge} e^{i_q{+}1}),\\
2\omega(e^{k}{\wedge}e^{k{+}1}) \in
\langle e^2, \underbrace{e^1, \dots, e^1}_{2k-3},e^2 \rangle,\\
(-1)^{i_1}\omega(e^{i_1}{\wedge} e^{i_2}{\wedge} {\dots}
{\wedge} e^{i_q}{\wedge} e^{i_q{+}1}) \in
\langle e^2, \underbrace{e^1, \dots, e^1}_{i_1-2},
\omega(e^{i_2}{\wedge}{\dots}
{\wedge} e^{i_q}{\wedge} e^{i_q{+}1})
\rangle,
\end{split}
\end{equation}
\end{theorem}

First of all we present a graded defining system (a graded
formal connection) $A$ for a Massey product
$\langle e^2, e^1, \dots, e^1,\omega(e^{i_2}{\wedge}{\dots}
{\wedge} e^{i_q}{\wedge} e^{i_q{+}1}) \rangle$.
To simplify the formulas we will write $\omega$ instead of
$\omega(e^{i_2}{\wedge}{\dots}
{\wedge} e^{i_q}{\wedge} e^{i_q{+}1})$.

One can verify that the following matrix $A$
with non-zero entries at the second diagonal, first line and
first row gives us an answer.

$$
A=\left(\begin{array}{ccccccc}
  0      & e^2 & -e^3 & e^4  & \dots & (-1)^{i_1}e^{i_1}   & 0  \\
  0      & 0      & e^1 & 0& \dots & 0   & D_{-1}^{i_1-2}\omega   \\
  0      & 0      & 0 & e^1 &  \dots & 0   & D_{-1}^{i_1-3}\omega   \\

   &      &        &        & \dots &            &           \\
  0  &0    & 0      & 0      & \dots & e^1 & D_{-1}\omega \\
  0   &0   & 0      & 0      & \dots & 0          & \omega   \\
  0    &0  & 0      & 0      & \dots & 0          & 0
  \end{array}\right).
$$
The proof follows from
$$
d(D_{-1}^k \omega)=e^1{\wedge}D_{-1}^{k{-}1} \omega, \quad
d((-1)^k e^k)=(-1)^{k{-}1}e^{k{-}1}{\wedge}e^1.
$$
The related cocycle $c(A)$ is equal to
$$
c(A)=\sum_{l\ge 2}^{i_1}(-1)^l e^{l} \wedge D_{-1}^{i_1{-}l} \omega
=(-1)^{i_1}\sum_{k\ge 0}^{i_1{-}2}(-1)^k D_1^k e^{i_1} \wedge D_{-1}^k \omega
$$

\begin{example}
We take $\langle e^2, e^1, \omega(e^4{\wedge} e^5)\rangle $.
$$
A=\left(\begin{array}{cccc}
  0      & e^2 & -e^3 &  0  \\
  0      & 0      & e^1 & D_{-1}\omega(e^4{\wedge} e^5)   \\
  0   &0    & 0          & \omega(e^4{\wedge} e^5)   \\
  0    &0  & 0      & 0
    \end{array}\right).
$$
Compute the related cocycle $c(A)$
\begin{multline}
c(A)=e^2\wedge D_{-1}\omega(e^4{\wedge} e^5)
-e^3 \wedge \omega(e^4{\wedge} e^5)=\\
=e^2\wedge (e^4{\wedge} e^6-2e^3{\wedge} e^7+3e^2{\wedge} e^8)
-e^3 \wedge (e^4{\wedge} e^5-e^3{\wedge} e^6+e^2{\wedge} e^7)=\\
=-e^3 {\wedge} e^4{\wedge} e^5+e^2 {\wedge} e^4{\wedge} e^6
-e^2 {\wedge} e^3{\wedge} e^7=-\omega(e^3 {\wedge} e^4{\wedge} e^5).
\end{multline}
\end{example}
Let us return to Theorem \ref{ThEfr}. Consider the question: for which sets $a_1, \dots, a_n$ of one-dimensional cocycles from $H^*({\mathfrak m}_0)$, all $k$-step Massey products $\langle a_1, \dots, a_n \rangle_k$ are defined and trivial? The following theorem answers this question.
\begin{theorem}[Millionshchikov~\cite{Mill3}]
\label{1_Massey}
Up to an equivalence
the following trivial $n$-fold Massey products of non-zero cohomology
classes from $H^1(\mathfrak{m}_0)$ are defined:

$$
\begin{tabular}{|c|c|c|}
\hline
&&\\[-10pt]
name &  Massey product & parameters\\
&&\\[-10pt]
\hline
&&\\[-10pt]
$A^{n+1}_{\lambda}$
 & $\underbrace{\langle \alpha e^1{+}\beta e^2,\alpha e^1{+}\beta e^2,
\dots,\alpha e^1{+}\beta e^2\rangle}_n$ &
$n \ge 3, \; \lambda=(\alpha,\beta) \in {\mathbb K}P^1$ \\
&&\\[-10pt]
\hline
&&\\[-10pt]
$B^{n+1}_{\alpha,\beta}$&
 $\underbrace{\langle \lambda_1 e^1{+}e^2,\lambda_2 e^1{+}e^2,
{\dots},\lambda_n e^1{+}e^2 \rangle}_n$ &
$\begin{array}{c}
n \ge 3, \; \lambda_i=i \alpha{+}\beta, \\
\alpha,\beta \in {\mathbb K}, \; \alpha \ne 0 \end{array}$
\\
&&\\[-10pt]
\hline
&&\\[-10pt]
$C^{n+1}_{l,\alpha}$ &
$\underbrace{\langle e^1, {\dots}, e^1}_l, e^2{+}\alpha e^1 ,
\underbrace{e^1,\dots, e^1\rangle}_{n-l-1}$ &
$\begin{array}{c} \alpha \in {\mathbb K},\; n \ge 3,\\
0\le l \le n{-}1 \end{array}$  \\
&&\\[-10pt]
\hline
&&\\[-10pt]
$D^{2k+3}_{\alpha,\beta}$&
 $\langle e^2+\alpha e^1, \underbrace{e^1,\dots,e^1}_{2k},
e^2+\beta e^1\rangle$
 & $k \ge 1, \;\alpha,\beta \in {\mathbb K}$ \\[2pt]
\hline
\end{tabular}
$$
\end{theorem}

It should be noted here that this result is related to the classification
of indecomposable
graded thread modules over $\mathfrak{m}_0$ from \cite{Ben}.


\section{Massey products in Koszul homology of local rings}
\label{local_rings}

Starting with this section, we turn to a discussion of triple and higher Massey products in Koszul homology of local rings and their applications in toric topology.

\begin{definition}\label{koszul}
Let $(A,\mathbf m,\ko)$ be a (commutative Noetherian) local ring, its unique maximal ideal $\mathbf m$ having minimal set of generators $(x_{1},\ldots,x_{m})$ and its residue field being $\ko=A/\mathbf m$.
A \emph{Koszul complex} of $(A,\mathbf m,\ko)$ is defined to be an exterior algebra $K_{A}=\Lambda\,A^{m}$, where $A^m$ denotes the free $A$-module generated by a set $\{e_{1},\ldots,e_{m}\}$, which is a differential graded algebra with a differential $d$ acting as follows:
$$
d(e_{i_1}\wedge\ldots\wedge e_{i_k})=\sum\limits_{r=1}^{k}(-1)^{r-1}x_{i_r}e_{i_1}\wedge\ldots\wedge \widehat{e_{i_{r}}}\wedge\ldots\wedge e_{i_k}.
$$ 
\end{definition}

\begin{definition}\label{poincare}
Let $(A,\mathbf m,\ko)$ be a local ring. Then for an $A$-module $M$ we define its \emph{Poincar\'e series} to be a formal power series of the type
$$
P_{A}(M;t)=\sum\limits_{i=0}^{\infty}\dim_{\ko}\Tor_{i}^{A}(M,\ko)t^{i}.
$$ 
By definition, the $\ko$-module $\Tor^{A}_{i}(M,\ko)$ is the $i$th homology of a projective resolution for $\ko$ (the latter viewed as an $A$-module via the quotient map $A\to A/\mathbf m=\ko$) tensored by $M$.\\

We call $P_{A}:=P_{A}(\ko;t)$ \emph{Poincar\'e series} of the local ring $A$.
\end{definition}

A classical problem of homological algebra was to prove a conjecture by Serre and Kaplansky, which asserted that $P_{A}$ is a rational function.
Anick~\cite{Anick82} (1982) found a counterexample being a quotient ring of a polynomial ring over an ideal generated by a set of monomials alongside with one binomial. 

On the other hand, in 1982 Backelin~\cite{Back82} proved the conjecture is true for \emph{monomial rings}, that is,   
a quotient ring of a polynomial ring by an ideal generated by monomials.
More precisely, a Poincar\'e series of a monomial ring 
$$
A=\ko[x_{1},\ldots,x_{m}]/I
$$
has a form
$$
P_{A}=\frac{(1+t)^{m}}{Q_{A,\ko}(t)}.
$$
It remains to be an open problem to determine effectively the coefficients of the denominator polynomial $Q_{A,\ko}(t)$ for various important classes of local rings.

In fact, homology of local rings $(A,\mathbf{m},\ko)$ and that of loop spaces $\Omega X$ are closely related. Namely, when the quotient field $\ko=\mathbb Q,$ or $\mathbb R$ and $\mathbf{m}^{3}=0$, Poincar\'e series of $A$ acquires topological interpretation as Poincar\'e series of a loop space for a certain simply connected 4-dimensional CW complex $X$. This enabled to solve the Serre--Kaplansky conjecture in both algebraic and topological settings at the same time.

The properties of the Poincar\'e series of a loop space $\Omega X$ on a finite simply connected CW complex $X$ were related to the homotopy properties of $X$, in particular, to the generating function for ranks of the homotopy groups of $X$, in a series of influential works by Babenko~\cite{Ba79,Ba80,Ba86}.
An extensive survey on the problem of rationality and growth for Poincar\'e series in the context of homological algebra and homotopy theory can be found in~\cite{Ba86}.
 
Using a spectral sequence associated with a presentation of a local ring as a quotient ring of a regular local ring, Serre showed that for any local ring $A$ the following coefficient-wise inequality for its Poincar\'e series holds:
\begin{equation}\label{serre}
P_{A}\leq\frac{(1+t)^m}{1-\sum\,b_{i}t^{i+1}}, 
\end{equation} 
where $m=\dim_{\ko}\mathbf{m}/\mathbf{m}^{2}$ and $b_{i}=\dim_{\ko}H_{i}(K_{A})$.

In 1962 Golod obtained a key result: he identified an important class of local rings with rational Poincar\'e series in terms of vanishing of multiplication and all Massey products in their Koszul homology.

\begin{theorem}[{\cite{Go}}]\label{golod}
Let $A$ be a local ring. Then Serre's inequality (\ref{serre}) turns into equality if and only if multiplication and all Massey products in $H_{+}(K_{A})$ are trivial.
\end{theorem}

In their monograph~\cite{GL} Gulliksen and Levin proposed to refer to the local rings with the above property as \emph{Golod rings}.

In view of the above mentioned correspondence between Poincar\'e series of local rings and that of loop spaces, Golod rings correspond to loop spaces over wedges of spheres of various dimensions. In this case, spectral sequence of a path--loop fibration degenerates, the Betti numbers grow as fast as possible, and the Poincar\'e series turns out to be equal to a fraction with denominator as in~(\ref{serre}), where $b_i$ are the Betti numbers of the base space.

\begin{example}[{\cite{Go}}]
Let $A$ be a free reduced nilpotent algebra, that is, a quotient ring 
$A_{n,r}=\frac{\ko[x_{1},\ldots,x_{n}]}{(x_{1},\ldots,x_{n})^{r}}$. Golod~\cite{Go} observed that multiplication and all Massey products are trivial in Koszul homology of $A_{n,r}$ and, furthermore, its Betti numbers are equal to $b_{i}=\binom{i+r-2}{r-1}\binom{n+r-1}{i+r-1}$. Therefore, by Theorem~\ref{golod}, $A_{n,r}$ is a Golod ring and 
$$
P_{A_{n,r}}=\frac{(1+t)^n}{1-\sum\limits_{i=1}^{n}\binom{i+r-2}{r-1}\binom{n+r-1}{i+r-1}t^{i+1}},
$$
which generalizes a computation of Poincar\'e series given by Kostrikin and Shafarevich~\cite{KoSh57}.  
\end{example}

From now on we concentrate mainly on the special class of monomial rings called face rings, or Stanley--Reisner rings of simplicial complexes, see Definition~\ref{SR}. It is well known that applying to any monomial ring a procedure called `polarization of a monomial ideal'~\cite{Hart,Hoch} results in a Stanley--Reisner ring (of a certain simplicial complex) carrying the same homological properties as the initial ring. To define this special and important class of monomial rings, we need to recall the notions of a simplicial complex and a simple polytope.

\begin{definition}\label{simple}
By an (\emph{abstract}) {\emph{simplicial complex}} on a vertex set $[m]=\{1,2,\ldots,m\}$ we mean a subset $K$ of $2^{[m]}$ such that if $\sigma\in K$ and $\tau\in\sigma$, then $\tau\in K$. The singleton elements of $K$ are called its \emph{vertices} and the \emph{dimension} of $K$ is defined to be one less than the maximal number of vertices in an element of $K$.

By a (\emph{convex}) $n$-{\emph{dimensional simple polytope}} $P$ with $m$ facets we mean a bounded intersection of $m$ halfspaces in $\R^n$ such that the supporting hyperplanes of those halfspaces are in general position. The latter condition is equivalent to saying that each vertex of $P$ is an intersection of exactly $n$ of its \emph{facets} (i.e. faces of codimension one).   
\end{definition}

\begin{example}
Let $P^*$ be the polytope combinatorially dual to a simple $n$-dimensional polytope $P$ with $m$ facets. Then $P^*$ is a convex hull of $m$ vertices, which are in general position in the ambient Euclidean space $\R^n$. Therefore, all proper faces of $P^*$ are simplices (we say, $P^*$ is a \emph{simplicial polytope}) and its boundary $K_P=\partial P^*$ is a simplicial complex of dimension $n-1$ with $m$ vertices. 
\end{example}

\begin{definition}\label{SR}
Let $\ko$ be a commutative ring with unit and $K$ be a simplicial complex on $[m]$. We call a quotient ring
$$
\ko[K]=\frac{\ko[v_{1},\ldots,v_{m}]}{(v_{i_1}\cdots v_{i_k}\,|\,\{i_{1},\ldots,i_{k}\}\notin K)},
$$
of a polynomial algebra over a monomial ideal a \emph{face ring}, or a \emph{Stanley--Reisner ring} of $K$. The monomial ideal $I=(v_{i_1}\cdots v_{i_k}\,|\,\{i_{1},\ldots,i_{k}\}\notin K)$ is called the \emph{Stanley--Reisner ideal} of $K$.
\end{definition}

By definition, $\ko[K]$ is a finite $\ko$-algebra and also a finitely generated $\ko[m]=\ko[v_{1},\ldots,v_{m}]$-module via the natural projection.
Now we are going to formulate a graded version of Definition~\ref{koszul}.

\begin{definition}\label{gradedkoszul}
Set $\mathrm{mdeg}(u_{i})=(-1;0,\ldots,2,\ldots,0)$, $\mathrm{mdeg}(v_{i})=(0;0,\ldots,2,\ldots,0)$ for $1\leq i\leq m$, and consider a $(0,1)$-vector $\mathbf{a}\in\mathbb{Z}_{2}^{m}$ (this grading will acquire its topological interpretation in Theorem~\ref{Hochtheorem} and Theorem~\ref{zkcoh} below).
Then a \emph{multigraded Tor-module} of $\ko[K]$ is a direct sum of $\ko$-modules
$$
\Tor_{\ko[m]}^{-i,2\mathbf{a}}(\ko[K],\ko)=H^{-i,2\mathbf{a}}[\ko[K]\otimes_{\ko}\Lambda[u_{1},\ldots,u_{m}],d],
$$
where in the latter differential graded algebra its differential $d$ acts as follows: $d(u_{i})=v_{i}$ and $d(v_i)=0$ for all $1\leq i\leq m$.
\end{definition}

Denote by $K_{I}:=K\cap 2^{I}$ the \emph{induced subcomplex} of $K$ on the vertex set $I\subset [m]$.
In the above notation, the next result due to Hochster~\cite{Hoch} describes the $\ko$-module structure of $\Tor_{\ko[m]}^{-i,2\mathbf{a}}(\ko[K],\ko)$ in terms of induced subcomplexes in $K$. 

\begin{theorem}[{\cite{Hoch}}]\label{Hochtheorem}
There is a $\ko$-module isomorphism
$$
\Tor_{\ko[m]}^{-i,2\mathbf{I}}(\ko[K],\ko)\cong\tilde{H}^{|I|-i-1}(K_I;\ko),
$$
where the $j$th component of the $(0,1)$-vector $\mathbf{I}$ is either 0 or 1, depending on whether or not $j\in [m]$ is an element of $I\subset [m]$, and $|I|$ denotes the cardinality of $I$. On the right hand side one has a reduced simplicial cohomology group of an induced subcomplex $K_I$ in $K$. 
\end{theorem}

Now let us interprete the Tor-algebra (Koszul homology) 
$$
K_{\ko[K]}\cong\oplus\Tor_{\ko[m]}^{-i,2\mathbf{a}}(\ko[K],\ko)
$$
of a Stanley--Reisner ring $\ko[K]$ of a simplicial complex $K$ as a cohomology ring of a certain finite CW complex $\mathcal Z_K$. The space $\zk$, called a \emph{moment-angle-complex} of $K$, is the main object of study in Toric Topology and is a particular case of the following general construction, which is of its own interest in modern homotopy theory, see~\cite{BBC}. 
 
Denote by 
$$
({\bf{X}},{\bf{A}})=(\underline{X}_i,\underline{A}_i):=\{(X_i,A_i)\}_{i=1}^{m}
$$ 
an ordered set of topological pairs. The case $X_{i}=X, A_{i}=A$ was firstly introduced in~\cite{bu-pa00-2} under the name of a \emph{K-power} and was then intensively studied and generalized in a series of more recent works, among which are~\cite{BBCG10,G-T13,IK}. 

\begin{definition}[{\cite{BBCG10}}]
A {\emph{polyhedral product}} over a simplicial complex $K$ on the vertex set $[m]$ is a topological space
$$
({\bf{X}},{\bf{A}})^K=\bigcup\limits_{I\in K}({\bf{X}},{\bf{A}})^I\subseteq\prod\limits_{i=1}^{m}\,X_{i},
$$
where $({\bf{X}},{\bf{A}})^I=\prod\limits_{i=1}^{m} Y_{i}$ and $Y_{i}=X_{i}$ if $i\in I$, and $Y_{i}=A_{i}$ if $i\notin I$.
\end{definition}
The term `polyhedral product' was suggested by Browder (cf.~\cite{BBCG10}).

\begin{example}
Suppose $X_{i}=X$ and $A_{i}=A$ for all $1\leq i\leq m$. Then the next spaces are particular cases of the above construction of a polyhedral product $({\bf{X}},{\bf{A}})^{K}=(X,A)^K$.
\begin{itemize}
\item[(1)] {\emph{Moment-angle-complex}} 
$\zk=(\mathbb{D}^2,\mathbb{S}^1)^K$;
\item[(2)] {\emph{Real moment-angle-complex}} 
$\mathcal R_K=([-1;1],\{-1,1\})^K$;
\item[(3)] \emph{Davis--Januzskiewicz space} $DJ(K):=\mathrm{E}\mathbb{T}^{m}\times_{\mathbb{T}^m}\zk\simeq(\mathbb{C}P^{\infty},*)^K$;
\item[(4)] Cubical complex $\cc(K)=(I^1,1)^K$ in the $m$-dimensional cube $I^{m}=[0,1]^m$, which is PL-homeomorphic to a cone over a barycentric subdivision of $K$.
\end{itemize} 
\end{example}

As was shown by Buchstaber and Panov~\cite{bu-pa00-2}, one has a commutative diagram
$$\begin{CD}
  \zk @>>>(\mathbb{D}^2)^m\\
  @VVrV\hspace{-0.2em} @VV\rho V @.\\
  \cc(K) @>i_c>> I^m
\end{CD}\eqno 
$$
where $i_{c}:\,\cc(K)\hookrightarrow I^{m}=(I^1,I^1)^{[m]}$ is an embedding of a cubical subcomplex, induced by an embedding of pairs: $(I^1,1)\subset (I^1,I^1)$, the maps $r$ and $\rho$ are projection maps onto the orbit spaces of a compact torus $\mathbb{T}^m$-action, induced by coordinatewise action of $\mathbb{T}^m$ on the complex unitary polydisk $(\mathbb{D}^2)^m$ in $\C^m$.

When $K=K_P$ is a boundary of an $n$-dimensional simplicial polytope with $m$ vertices, see Definition~\ref{simple}, Buchstaber and Panov~\cite{bu-pa00-2} proved $\zk$ is equivariantly homeomorphic to a smooth compact 2-connected $(m+n)$-dimensional manifold $\zp$ with a compact torus $\mathbb{T}^{m}$-action. They called it a \emph{moment-angle manifold} of the simple $n$-dimensional polytope $P$ with $m$ facets. This space appeared firstly in the paper by Davis and Januszkiewicz~\cite{DJ}. 

The following definition of $\zp$ given by Buchstaber, Panov, and Ray~\cite{BPR} is equivalent to the definition of $\zp$ from~\cite{DJ}. Its advantage is that it yields a realization of $\zp$ as a complete intersection of Hermitian quadrics in $\mathbb C^m$.

\begin{definition}
A \emph{moment-angle manifold} $\mathcal Z_P$ of a polytope $P$ is a pullback defined from the following commutative diagram
$$\begin{CD}
  \mathcal Z_P @>i_Z>>\C^m\\
  @VVV\hspace{-0.2em} @VV\mu V @.\\
  P @>i_P>> \R^m_\ge
\end{CD}\eqno 
$$
where $i_P$ is an affine embedding of $P$ into the nonnegative orthant $\R^m_\ge=\{x=(x_{1},\ldots,x_{m})\,|\,x_{i}\geq 0\,\text{ for all }1\leq i\leq m\}$, and $\mu(z_1,\ldots,z_m)=(|z_1|^2,\ldots,|z_m|^2)$. The projection map $\zp\rightarrow P$ in the above diagram is a quotient map of the canonical $\mathbb{T}^m$-action on $\zp$, induced by the standard coordinatewise action of
\[
  \mathbb T^m=\{\mb z\in\C^m\colon|z_i|=1\quad\text{for }i=1,\ldots,m\}
\]
on~$\C^m$. Therefore, $\mathbb T^m$ acts on $\zp$ with an orbit space $P$ and $i_Z$ is a $\mathbb T^m$-equivariant embedding.
\end{definition}

Moreover, it was proved in~\cite{bu-pa00-2} that there exists a homotopy fibration of polyhedral products
$$
\zk\to(\mathbb{C}P^{\infty},\ast)^{K}\to BT^{m},
$$
where the first map is induced by a natural map of pairs $(\mathbb{D}^{2},\mathbb{S}^{1})\to(\mathbb{C}P^{1},\ast)$ and the second is induced by inclusion $(\mathbb{C}P^{\infty},\ast)\hookrightarrow (\mathbb{C}P^{\infty},\mathbb{C}P^{\infty})$.

Applying Eilenberg--Moore spectral sequence argument to the above homotopy fibration and analyzing topology of the polyhedral products involved, Baskakov, Buchstaber and Panov~\cite{BBP} obtained the next fundamental result, which links toric topology with combinatorial commutative algebra and, in particular, with homology theory of face rings. 

\begin{theorem}[{\cite[Theorem 4.5.4]{TT}}]\label{zkcoh}
Cohomology algebra of $\zk$ over a commutative ring with unit $\ko$ is given by isomorphisms
\[
\begin{aligned}
  H^{*}(\mathcal Z_K;\ko)\cong\Tor_{\ko[v_1,\ldots,v_m]}^{*,*}(\ko[K],\ko)\cong H^{*,*}\bigl[R(K),d\bigr]\cong \bigoplus\limits_{I\subset [m]}\widetilde{H}^{*}(K_{I};\ko),
\end{aligned}
\]
where the differential (multi)graded algebra  $R(K):=\ko[K]\otimes_{\ko}\Lambda[u_{1},\ldots,u_{m}]/(v_{i}^{2}=u_{i}v_{i}=0,1\leq i\leq m)$ and $d(u_i)=v_i, d(v_i)=0$. Here we denote by $\widetilde{H}^{*}(K_{I};\ko)$ reduced simplicial cohomology of a simplicial complex $K_{I}$ and set $\tilde{H}^{-1}(\varnothing;\ko)\cong\ko$. The last isomorphism above is a sum of isomorphisms 
$$
H^{p}(\zk;\ko)\cong\sum\limits_{I\subset [m]}\widetilde{H}^{p-|I|-1}(K_{I};\ko).
$$

In order to determine a product of two cohomology classes $\alpha=[a]\in\tilde{H}^{p}(K_{I_1};\ko)$ and $\beta=[b]\in\tilde{H}^{q}(K_{I_2};\ko)$, consider an embedding of simplicial complexes $i:\,K_{I_{1}\sqcup I_{2}}\rightarrow K_{I_1}*K_{I_2}$ and a canonical $\ko$-module isomorphism of cochains:
$$
j:\,\tilde{C}^{p}(K_{I_1};\ko)\otimes\tilde{C}^{q}(K_{I_2};\ko)\rightarrow\tilde{C}^{p+q+1}(K_{I_{1}}*K_{I_2};\ko).
$$
Then a product of the classes $\alpha$ and $\beta$ is given by:
$$
\alpha\cdot\beta=\begin{cases}
0,&\text{if $I_{1}\cap I_{2}\neq\varnothing$;}\\
i^{*}[j(a\otimes b)]\in\tilde{H}^{p+q+1}(K_{I_{1}\sqcup I_{2}};\ko),&\text{if $I_{1}\cap I_{2}=\varnothing$.}
\end{cases}
$$
\end{theorem}

It turned out that Golodness of a face ring $\ko[K]$ is closely related to the case when $\zk$ is homotopy equivalent to a wedge of spheres; by Theorem~\ref{zkcoh}, the latter implies the former. However, the opposite statement is not true, see~\cite{G-P-T-W,L2015,I-Y}. 

Using methods of stable homotopy theory and toric topology such as the stable homotopy decomposition of polyhedral products due to Bahri, Bendersky, Cohen, and Gitler~\cite{BBCG10,BBCG14,BBCG15} and the fat wedge filtration method due to Iriye and Kishimoto~\cite{IK14}, several authors, among which are Grbi\'c and Theriault~\cite{GT}, Iriye and Kishimoto~\cite{IK}, Grbi\'c, Panov, Theriault, and Wu~\cite{G-P-T-W}, were able to identify a number of important classes of simplicial complexes $K$ for which $\ko[K]$ is a Golod ring over any field $\ko$. In the latter case $K$ is called a \emph{Golod simplicial complex}. For a detailed exposition on this problem we refer the reader to a survey article by Grbi\'c and Theriault~\cite{GT2}.

By definition, Golodness of $\ko[K]$ implies triviality of multiplication in its Koszul homology $H_{+}(K_{\ko[K]})$, or equivalently, $\mathrm{cup}(\zk)=1$, where $\mathrm{cup}(X)$ denotes \emph{cohomology length} of a space $X$, that is, the maximal number of classes of positive dimension in $H^*(X;\ko)$ having a nonzero product. Another related area of research, which also attracts much attention from the scholars these days, emerged from a false claim made in~\cite{BJ}: Golodness of $K$ is equivalent to $\mathrm{cup}(\zk)=1$. The first counterexample was constructed by Katth\"an in 2015. More precisely, he proved the next result.     

\begin{theorem}[{\cite{Kat}}]
The following statements hold.
\begin{itemize}
\item[(1)] If $\dim K\leq 3$, then $\mathrm{cup}(\zk)=1$ $\Longleftrightarrow$ $K$ is a Golod simplicial complex;
\item[(2)] There exists a 4-dimensional simplicial complex $K$ such that
\begin{itemize}
\item[(a)] $\mathrm{cup}(\zk)=1$;
\item[(b)] There is a non-trivial triple Massey product $\langle\alpha_{1},\alpha_{2},\alpha_{3}\rangle\subset H^*(\zk)$; therefore, $K$ is not a Golod complex.
\end{itemize}
\end{itemize}
\end{theorem}

Since that time there have already appeared several papers devoted to identification of the class of simplicial complexes $K$ for which Golodness of $K$ is equivalent to triviality of multiplication in $H^{+}(\zk)$, that is, to $\mathrm{cup}(\zk)=1$. Such simplicial complexes and face rings are refered to as {\emph{quasi-Golod}}, that is the case when there are no non-trivial triple or higher order Massey products in their Koszul homology. 

To state the most general known result characterizing a class of Golod (and quasi-Golod) simplicial complexes and face rings we need the following two notions. A simplicial complex $K$ is called \emph{flag} if all its minimal non-faces (with respect to inclusion relation) have exactly two vertices.
A simple graph $\Gamma$ is \emph{chordal} if it does not have any induced cycles of length greater than three.

\begin{theorem}[{\cite{G-P-T-W,P-V}}]
Let $K$ be a flag simplicial complex. 
The following statements are equivalent:
\begin{itemize}
\item[(1)] The 1-skeleton $\sk^{1}(K)$ of $K$ is a chordal graph;
\item[(2)] $\mathrm{cup}(\zk)=1$;
\item[(3)] $\zk$ is homotopy equivalent to a wedge of spheres;
\item[(4)] $K$ is a Golod complex;
\item[(5)] Commutator subgroup $\pi_{1}(\mathcal R_K)=\RC_{K}^{'}$ of the right-angled Coxeter group $\RC_{K}$ is a free group;    
\item[(6)] Associated graded Lie algebra $\gr(\RC_{K}^{'})$ is free.
\end{itemize}
\end{theorem}

There is an extremely important class of topological spaces arising in homotopy theory, for which the rational homotopy type is determined by the singular cohomology ring of a space (over rationals). Namely, a space $X$ is called \emph{rationally formal} if its Sullivan-de Rham algebra $[A,d]$ of PL-forms with coefficients in $\mathbb Q$ is formal in the category of commutative differential graded algebras (CDGA), i.e., there exists a zigzag of quasi-isomorphisms (weak equivalence) between $[A,d]$ and its cohomology algebra $[H^{*}(A),0]$. 

Several important classes of topological spaces and smooth manifolds were proved to be formal, among them are spheres, H-spaces, Eilenberg--MacLane spaces $K(\pi,n)$ for $n>1$, compact connected Lie groups $G$ and their classifying spaces $BG$. Moreover, formality is preserved by wedges, direct products, and connected sums. 

Polyhedral products of the type $(X,\ast)^{K}$ are formal spaces (over $\ko=\mathbb Q$), provided the space $X$ is formal (over $\mathbb Q$), see~\cite[Chapter 8]{TT}. In particular, Davis--Januszkiewicz spaces $\mathrm{DJ}(K)$ are formal for any simplicial complex $K$.

It is not hard to see that formality of $X$ implies all triple and higher Massey products vanish in $H^{*}(X)$; therefore, Massey products serve as an obstruction to formality of a differential graded algebra, or to that of a topological space. The next section is devoted to the case, when there exist non-trivial Massey products in (integral) cohomology of moment-angle-complexes and moment-angle manifolds.


\section{Massey products in Toric Topology}

Since for any simplicial complex $K$ its moment-angle-complex $\zk$ is 2-connected, the lowest possible dimension of cohomology classes in $H^{+}(\zk)$, which can form a non-trivial Massey product is equal to three. 
Note that by a result of Halperin and Stasheff~\cite{HS}, if a space $X$ is formal, then $H^*(X)$ is generated by spherical classes. 

Although all 3-dimensional classes in $H^*(\zk)$ are spherical, which can be easily seen from Theorem~\ref{zkcoh}, there exists a wide class of simplicial complexes $K$ and simple polytopes $P$ such that $H^*(\zk)$ and $H^*(\zp)$ contain non-trivial triple Massey products of 3-dimensional classes. The case of non-trivial triple Massey products of 3-dimensional classes in $H^*(\zk)$ was analyzed by Denham and Suciu~\cite{D-S} (strictly defined products) and recently by Grbi\'c and Linton~\cite{GL1} (products with non-zero indeterminacy). 
When a simple polytope $P$ is a graph-associahedron, necessary and sufficient conditions for $H^*(\zp)$ to contain a non-trivial triple Massey product of 3-dimensional classes were obtained by the first author in~\cite{L2017}.

The first example of a non-trivial triple Massey product in cohomology of a 
polyhedral product was given by Baskakov~\cite{BaskM}. His construction was later generalized to higher order Massey products by the first author~\cite{L2016,L2017,L2019}, by Buchstaber and Limonchenko~\cite{BL}, who developed a theory of {\emph{direct families of polytopes with non-trivial Massey products}} (it is beyond the scope of this survey), and recently by Grbi\'c and Linton~\cite{GL2}. The current section is devoted to a discussion of the above mentioned results.

At first, let us consider the case of a non-trivial triple Massey product of 3-dimensional classes in $H^*(\zk)$. Recall that by Theorem~\ref{zkcoh}, a 3-dimensional element in $H^*(\zk)$ corresponds to a pair of vertices of $K$ not connected by an edge, that is, to an induced subcomplex $K_I$ in $K$, where $|I|=2$ and $I\notin K$.

In fact, the next criterion shows that non-trivial triple Massey products
$\langle\alpha_{1},\alpha_{2},\alpha_{3}\rangle\subset\tilde{H}^*(\zk)$ with $\dim\alpha_{i}=3,1\leq i\leq 3$ are determined by the \emph{graph} (i.e., the 1-skeleton) $\sk^{1}(K)$ of the simplicial complex $K$.  

\begin{figure}[h]
\includegraphics[scale=0.5]{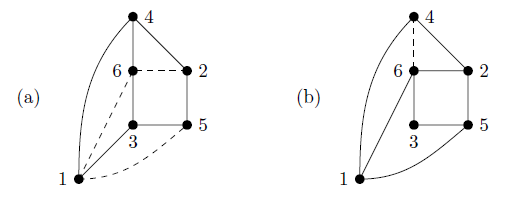}
\caption{Non-trivial triple Massey products of 3-dimensional classes in $H^*(\zk)$}\label{Masseyfig}
\end{figure}

\begin{theorem}[{\cite{D-S,GL1}}]\label{tripleMassey}
Let $\langle\alpha_{1},\alpha_{2},\alpha_{3}\rangle\subset\tilde{H}^*(\zk)$ with $\dim\alpha_{i}=3,1\leq i\leq 3$ be a defined Massey product. Then
\begin{itemize}
\item[(a)] $\langle\alpha_{1},\alpha_{2},\alpha_{3}\rangle$ is non-trivial and strictly defined if and only if $K$ contains one of the six different induced subgraphs described in Figure~\ref{Masseyfig}~(a) and $\alpha_{i}$ is a generator of the group $\tilde{H}^{0}(K_{\{i,i+1\}})$ for $1\leq i\leq 3$, cf. Theorem~\ref{zkcoh};
\item[(b)] $\langle\alpha_{1},\alpha_{2},\alpha_{3}\rangle$ is non-trivial and has non-zero indeterminacy if and only if $K$ contains one of the two different induced subgraphs described in Figure~\ref{Masseyfig}~(b) and $\alpha_{i}$ is a generator of the group $\tilde{H}^{0}(K_{\{i,i+1\}})$ for $1\leq i\leq 3$, cf. Theorem~\ref{zkcoh}.
\end{itemize}
\end{theorem}

\begin{remark}
Statement (a) was proved firstly by Denham and Suciu as~\cite[Theorem 6.1.1]{D-S}. However, it was asserted there that any non-trivial triple Massey product of 3-dimensional classes in $H^*(\zk)$ has such a form; the additional two graphs of Figure~\ref{Masseyfig}~(b) were found and the statement (b) above was proved by Grbi\'c and Linton~\cite{GL1}. The result of~\cite{GL1} provided us with a first example of a non-strictly defined non-trivial Massey product in cohomology of a moment-angle-complex. 
\end{remark}

Now we turn to a discussion of higher order non-trivial Massey products in Koszul homology of a Stanley--Reisner ring $H_{*}(K_{\ko[K]})\cong H^*(\zk;\ko)$. Following~\cite{L2019}, we set up the notation as indicated below.

Let us fix a set of induced subcomplexes $K_{I_j}$ in $K$ on pairwisely disjoint subsets of vertices $I_j\subset [m]$ for $1\leq j\leq k$ and their cohomology classes $\alpha_{j}\in\tilde{H}^{d(j)}(K_{I_j})$ of certain dimensions $d(j)\geq 0$ for $1\leq j\leq k$. If an $s$-fold Massey product ($s\leq k$) of consecutive classes $\langle\alpha_{i+1},\ldots,\alpha_{i+s}\rangle$ for $1\leq i+1<i+s\leq k$ is defined, then $\langle\alpha_{i+1},\ldots,\alpha_{i+s}\rangle$ is a subset of
$\tilde{H}^{d(i+1,i+s)}(K_{I_{i+1}\sqcup\ldots\sqcup {I_{i+s}}})$,
where $d(i+1,i+s):=d(i+1)+\ldots+d(i+s)+1$. 

Our goal is to determine the conditions sufficient for a Massey product $\langle\alpha_{1},\ldots,\alpha_{k}\rangle$ of cohomology classes introduced above to be defined and, furthermore, to be strictly defined. 

We are going to present a complete proof of the following theorem, in which statements (1) and (2) were originally proved by the first author as~\cite[Lemma 3.3]{L2019}. Alongside with statement (3), this result provides an effective tool to determine non-trivial $k$-fold Massey products for $k\geq 3$ in $H^*(\zk)$ for a given simplicial complex $K$, when one knows {\emph{multigraded}} (or, \emph{algebraic}) \emph{Betti numbers} of $K$ (i.e., the dimensions of the (multi)graded components of the Tor-module of $\ko[K]$) and the combinatorial structure of the corresponding induced subcomplex $K_{I_{1}\sqcup\ldots\sqcup {I_k}}$.
 
\begin{theorem}\label{mainlemma}
Let $k\geq 3$. Then 
\begin{itemize}
\item[(1)] If $\tilde{H}^{d(s,r+s)}(K_{I_{s}\sqcup\ldots\sqcup{I_{r+s}}})=0,1\leq s\leq k-r,1\leq r\leq k-2$, then the $k$-fold Massey product $\langle\alpha_{1},\ldots,\alpha_{k}\rangle$ is defined;
\item[(2)] If $\tilde{H}^{d(s,r+s)-1}(K_{I_{s}\sqcup\ldots\sqcup{I_{r+s}}})=0,1\leq s\leq k-r,1\leq r\leq k-2$ and the $k$-fold Massey product $\langle\alpha_{1},\ldots,\alpha_{k}\rangle$ is defined, then the $k$-fold Massey product $\langle\alpha_{1},\ldots,\alpha_{k}\rangle$ is strictly defined.
\item[(3)] In the latter case, there exists a defining system $C=(c_{i,j})_{i,j=1}^{k+1}$ for $\langle\alpha_{1},\ldots,\alpha_{k}\rangle$ such that 
$$
c_{s,r+s+1}\in C^{d(s,r+s)-1}(K_{I_{s}\sqcup\ldots\sqcup {I_{r+s}}}),
1\leq s\leq k-r,1\leq r\leq k-2
$$
and
$$
\langle\alpha_{1},\ldots,\alpha_{k}\rangle=\{[a(C)]\}\in H^{d(1,k)}(K_{I_{1}\sqcup\ldots\sqcup {I_{k}}}),\quad a(C)=-\sum\limits_{j=1}^{k-1}\bar{c}_{1,1+j}\wedge c_{1+j,k+1},
$$ 
in the differential graded algebra $\oplus_{I\subset [m]} C^{*}(K_I)\cong C^*(\zk)$.
\end{itemize}
\end{theorem}
\begin{proof}
The proof goes by induction on $k\geq 3$. 
First, we prove statement (1).

For $k=3$ the condition (1) implies that the 2-fold products $\langle\alpha_{1},\alpha_{2}\rangle$ and $\langle\alpha_{2},\alpha_{3}\rangle$ vanish and the triple Massey product $\langle\alpha_{1},\alpha_{2},\alpha_{3}\rangle$ is defined. By the condition (2), indeterminacy in $\langle\alpha_{1},\alpha_{2},\alpha_{3}\rangle$ is trivial, and therefore, $\langle\alpha_{1},\alpha_{2},\alpha_{3}\rangle$ is strictly defined. 

The inductive hypothesis follows, since all the Massey products of consecutive elements of orders $2,\ldots,k-1$ for $k\geq 4$ are defined by the inductive assumption and contain only zero elements by Theorem~\ref{zkcoh}. Therefore, a defining system $C$ exists and the corresponding cocycle $a(C)$ represents an element in the Massey product $\langle\alpha_{1},\ldots,\alpha_{k}\rangle$.

Now we prove statement (2).

We need to show that $[a(C)]=[a(C')]$ for any two defining systems $C$ and $C'$. By the inductive assumption, suppose that the statement holds for defined higher Massey products of orders less than $k\geq 4$. The induction step can be divided into two parts, and in both of them we also act by induction.
\\
I. Let us construct a sequence of defining systems $C(1),\ldots,C(k-1)$ for the defined Massey product $\langle\alpha_{1},\ldots,\alpha_{k}\rangle$ such that the following properties:
\begin{itemize}
\item[(1)] $C(1)=C$;
\item[(2)] $c_{ij}(r)=c'_{ij}$, if $j-i\leq r$;
\item[(3)] $[a(C(r))]=[a(C(r+1))]$, for all $1\leq r\leq k-2$.
\end{itemize}
Observe that (2) implies $c_{i,i+1}(r)=a_{i}=a'_{i}$ for all $1\leq i\leq k$, and $C(k-1)=C'$. We apply induction on $r\geq 1$ to determine the defining systems $C(r)$. Since $C(1)=C$ by (1), we need to prove the induction step assuming that $C(r)$ is already defined. 

For any $1\leq s\leq k-r$ let us consider a cochain 
$$
b_{s}=c'_{s,r+s+1}-c_{s,r+s+1}(r).
$$
By definition of a Massey product, 
$$
d(b_{s})=d(c'_{s,r+s+1})-d(c_{s,r+s+1}(r))=\sum\limits_{p=s+1}^{r+s}\bar{c}'_{s,p}\wedge c'_{p,r+s+1}-
\sum\limits_{p=s+1}^{r+s}\bar{c}_{s,p}(r)\wedge c_{p,r+s+1}(r)=0
$$ 
by property (2) of $C(r)$ above. It follows that $b_{s}$ is a cocycle, and therefore, 
$$
[b_s]\in\tilde{H}^{d(s,r+s)-1}(K_{I_{s}\sqcup\ldots\sqcup{I_{r+s}}})=0,
$$ 
for all $1\leq s\leq k-r$, by condition (2) of our statement, since here we have $1\leq r\leq k-2$ by property (3) above.  
\\
II. Observe that the construction of the defining system $C(r+1)$ will be finished if one is able to determine a sequence of defining systems $C(r,s)$ ($0\leq s\leq k-r$) for $\langle\alpha_{1},\ldots,\alpha_{k}\rangle$ with the following properties:
\begin{itemize}
\item[(1')] $C(r,0)=C(r)$;
\item[(2')] $c_{ij}(r,s)=c_{ij}(r)$, if $j-i\leq r$, and
$$
c_{ij}(r,s)=\begin{cases}
c_{ij}(r),&\text{if $i>s$; (*)}\\
c_{ij}(r)+b_{i},&\text{if $i\leq s$ (**)}
\end{cases}
$$
when $j-i=r+1\geq 2$;
\item[(3')] $[a(C(r,s))]=[a(C(r,s+1))]$, for all $0\leq s\leq k-r-1$. 
\end{itemize}
It is easy to see that property (2')(**) for $j=i+r+1$ implies that $c_{ij}(r,k-r)=c_{ij}(r)+(c'_{i,r+1+i}-c_{i,r+1+i}(r))=c'_{ij}$, the latter being equal to $c_{ij}(r+1)$ by property (2) above, and therefore, one can take $C(r+1)=C(r,k-r)$ and the proof will be completed by induction.

So, to finish the proof, it suffices to construct a sequence of defining systems $C(r,s)$. Now we shall do it acting by induction on $s\geq 0$. The base of induction $s=0$ follows by property (1'). Now assume that we have already constructed $C(r,s)$ and let us determine the defining system $C(r,s+1)$.

If $i>s+1$, then one can take $c_{ij}(r,s+1)=c_{ij}(r,s)$, see (2')(*). Similarly, one can also take $c_{ij}(r,s+1)=c_{ij}(r,s)$ if $j<s+r+2$, see (2')(**). Suppose $1\leq i\leq s+1<s+2+r\leq j\leq k+1$. Now, by induction on $j-i\geq r+1$ we shall determine a set of cochains $\{b_{ij}\}$ such that
$$
c_{ij}(r,s+1)=c_{ij}(r,s)+b_{ij}\eqno (***).
$$

Equality (**) implies that $c_{s+1,r+2+s}(r,s+1)=c_{s+1,r+2+s}(r)+b_{s+1}$ and the latter is equal to $c_{s+1,r+2+s}(r,s)+b_{s+1}$ by equality (*). Therefore, one can set $b_{s+1,r+2+s}=b_{s+1}$. By inductive assumption, assume that for all $r+1\leq j-i<w$ the cochains $b_{ij}$ have already been defined. Then equality (***) implies that 
$$
d(b_{ij})=d(c_{ij}(r,s+1))-d(c_{ij}(r,s))=\sum\limits_{p=i+1}^{j-1}(\bar{c}_{i,p}(r,s)+\bar{b}_{i,p})\wedge (c_{p,j}(r,s)+b_{p,j})-
$$
$$
-\sum\limits_{p=i+1}^{j-1}\bar{c}_{i,p}(r,s)\wedge c_{p,j}(r,s)=\sum\limits_{p=i+1}^{s+1}\bar{c}_{i,p}(r,s)\wedge b_{p,j}+
\sum\limits_{p=r+s+2}^{j-1}\bar{b}_{i,p}\wedge c_{p,j}(r,s),
$$ 
where the last equality holds, since $\sum\limits_{p=i+1}^{j-1}\bar{b}_{i,p}\wedge b_{p,j}=0$, because $b_{p,j}=0$ when $p>s+1$, and $b_{i,p}=0$ when $p<r+2+s$, and one gets the following equality for any $r+1\leq j-i<w$:
$$
d(b_{ij})=\sum\limits_{p=i+1}^{s+1}\bar{c}_{i,p}(r,s)\wedge b_{p,j}+
\sum\limits_{p=r+s+2}^{j-1}\bar{b}_{i,p}\wedge c_{p,j}(r,s)\eqno (1.1)
$$   

Observe that for $j-i=w$ the right hand side of the equality (1.1) is a cocycle $a$ representing an element $\alpha=[a]$ in 
$$
-\langle\alpha_{i},\ldots,\alpha_{s},[b_{s+1}],\alpha_{s+r+2},
\ldots\alpha_{j-1}\rangle \eqno (1.2)
$$ 
This can be shown acting by induction on $j-i\geq r+2$. 
Indeed, for $j-i=r+2$ and $1\leq i\leq s+1<s+2+r\leq j\leq k+1$ we have only two possible cases: (1) $i=s+1,j=s+r+3$ and the right hand side of (1.1) has the form $\bar{b}_{s+1,r+s+2}\wedge c_{s+r+2,s+r+3}=\bar{b}_{s+1}a_{s+r+2}$. The latter cocycle represents $-\langle[b_{s+1}],\alpha_{s+r+2}\rangle$; (2) $i=s,j=s+r+2$ and the right hand side of (1.1) has the form $\bar{c}_{s,s+1}b_{s+1,s+r+2}=\bar{a}_{s}\wedge b_{s+1}$. The latter cocycle represents $-\langle\alpha_{s},[b_{s+1}]\rangle$. The induction step follows from the equality (1.1), definition of a (higher) Massey product, and the inductive assumption.   

Since $[b_{s+1}]=0$, one concludes that the Massey product given by the formula (1.2) above is trivial. Furthermore, as $r\geq 1$ it follows that the order of this Massey product is less than $(j-1)-i+1=j-i\leq k$. Therefore, we can apply the inductive assumption on $k$ to this Massey product, since 
$[b_{s+1}]\in\tilde{H}^{\beta}(K_{I_{s+1}\sqcup\ldots\sqcup{I_{r+s+1}}})$ for $\beta=d(s+1,r+s+1)-1=d(s+1)+\ldots+d(r+s+1)$. Thus, by the inductive assumption on $k$, we obtain that the Massey product
$$
0\in\langle\alpha_{i},\ldots,\alpha_{s},[b_{s+1}],\alpha_{s+r+2},
\ldots,\alpha_{j-1}\rangle
$$ 
is strictly defined, that is, it contains only zero. It follows that equality (1.1) above has a solution for $j-i=w$. Therefore, for all $1\leq i<j\leq k+1$ equality (1.1) implies 
$$
d(b_{ij})=\sum\limits_{p=i+1}^{j-1}\bar{c}_{i,p}(r,s+1)\wedge c_{p,j}(r,s+1)-
\sum\limits_{p=i+1}^{j-1}\bar{c}_{i,p}(r,s)\wedge c_{p,j}(r,s).
$$
The above equality means that: (i) $C(r,s+1)$ is also a defining system for the Massey product $\langle\alpha_{1},\ldots,\alpha_{k}\rangle$ and (ii) $[a(C(r,s+1))]=[a(C(r,s))]$ (when in the above formula $j-i=k$).
The whole proof of the statement (2) is now completed by induction on the order $k$ of a Massey product.

Finally we prove statement (3).

We proceed by induction on $k$ again, using the fact that the right hand side in Theorem~\ref{zkcoh} admits a multigraded refinement, that is, the differential $d$ respects the multigraded structure on the Tor-module, given by Definition~\ref{gradedkoszul} and Theorem~\ref{Hochtheorem}. Therefore, induction on $j-i=r\geq 1$ in the defining system $C$, see condition (3) above, gives us simplicial cochains $c_{s,r+s+1}\in C^{d(s,r+s)-1}(K_{I_{s}\sqcup\ldots\sqcup {I_{r+s}}})$ for all $1\leq s\leq k-r,1\leq r\leq k-2$, satisfying the relations for elements of a defining system themselves and giving the unique element $[a(C)]$ of the $k$-fold Massey product $\langle\alpha_{1},\ldots,\alpha_{k}\rangle$. The latter class is an element of the group $H^{d(1,k)}(K_{I_{1}\sqcup\ldots\sqcup {I_{k}}})$ by definition of multiplication in $H^*(\zk)$, see Theorem~\ref{zkcoh}. This finishes the proof of the theorem.
\end{proof}

\begin{remark}
It is easy to see that Theorem~\ref{mainlemma} implies the triple Massey products in~\cite{BaskM} and~\cite{D-S} are all strictly defined. On the other hand, in the example of a trivial triple Massey product in $H^*(\zp)$ when $P$ is a hexagon, see~\cite[Example 3.4.1]{L2019} as well as in the case of Theorem~\ref{tripleMassey} (b) the condition (2) of Theorem~\ref{mainlemma} is not satisfied and those Massey products are not strictly defined.     
\end{remark}


First examples of non-trivial higher Massey products of any order in $H^*(\zk)$ were constructed by Limonchenko in a short note~\cite{L2016}.
A complete proof of nontriviality and an example of computation of a non-trivial 4-fold Massey product were given by the first author in~\cite{L2017}. We describe this construction below, following~\cite{BL}.

\begin{definition}{\cite{L2016,BL}}\label{MasseyConstr}
Let $Q^0$ be a point and $Q^1\subset\R^1$ be a segment $[0,1]$. Denote by $I^{n}=[0,1]^n, n\geq 2$ the standard $n$-dimensional cube with facets $F_{1},\ldots,F_{2n}$ labeled in such a way that $F_{i},1\leq i\leq n$ contains the origin 0, $F_{i}$ and $F_{n+i}$ are parallel for all $1\leq i\leq n$. 
Then the face ring of the cube $I^n$ has the form:
$$
\ko[I^n]=\ko[v_{1},\ldots,v_{n},v_{n+1},\ldots,v_{2n}]/I_{I^n},
$$
where the Stanley--Reisner ideal is $I_{I^n}=(v_{1}v_{n+1},\ldots,v_{n}v_{2n})$.

Consider a polynomial ring
$$
\ko[v_{1},\ldots,v_{2n},v_{k',n+k'+i'}|\,1\leq i'\leq n-2, 1\leq k'\leq n-i']
$$
and its monomial ideal, generated by square free monomials:
$$
I=(v_{k}v_{n+k+i},v_{k',n+k'+i'}v_{n+k'+l},v_{k',n+k'+i'}v_{p},v_{k',n+k'+i'}v_{k'',n+k''+i''}),
$$
where $v_{j}$ corresponds to $F_{j}$ for all $1\leq j\leq 2n$, and 
$$
0\leq i\leq n-2, 1\leq k\leq n-i, 1\leq i',i''\leq n-2, 1\leq k'\leq n-i', 
$$
$$
1\leq k''\leq n-i'', 1\leq p\neq k'\leq k'+i', 0\leq l\neq i'\leq n-2, 
$$
$$
k'+i'=k''\,\text{or }k''+i''=k'.
$$

Let us define $Q^n\subset\R^n$ to be a simple polytope such that for its Stanley--Reisner ideal: $I_{Q^n}=I$. Note that $Q^n$ has a natural realization as a {\emph{2-truncated cube}}, that is, a simple $n$-dimensional polytope obtained from an $n$-dimensional cube as a result of performing truncations of faces of codimension two by generic hyperplanes in $\R^n$, see~\cite{BV}. Moreover, its combinatorial type does not depend on the order in which the faces of the cube $I^n$ are truncated (the generators $v_{i,j}$ correspond to the truncated faces $F_{i}\cap F_{j}$ of $I^n$).
\end{definition}

Then by~\cite[Theorem 3.6]{L2019}, for any $n\geq 2$ there exists a strictly defined non-trivial $n$-fold Massey product $\langle\alpha_{1},\ldots,\alpha_{n}\rangle\subset H^*(\mathcal Z_{Q^n})$, where $\alpha_{i}$ is a generator of $\tilde{H}^{0}(F_{i}\sqcup F_{n+i})$.

Finally, using the above construction, Theorem~\ref{mainlemma}, and the {\emph{simplicial multiwedge}} operation (or, $J$-{\emph{construction}}), introduced in the framework of toric topology by Bahri, Bendersky, Cohen, and Gitler~\cite{BBCG10}, the first author proved the following result.

\begin{theorem}[{\cite{L2019}}]\label{mymainMassey}
For any $n\geq 2$ there exists a strictly defined non-trivial Massey product of order $k$ in $H^*(\mathcal Z_{Q^n})$ for all $k$, $2\leq k\leq n$.
Furthermore, there exists a family of moment-angle manifolds $\mathcal{F}$ such that for any given $l,r\geq 2$ there is an $l$-connected manifold $M\in\mathcal{F}$ with a strictly defined non-trivial $r$-fold Massey product in $H^*(M)$.
\end{theorem}

Applying the theory of nestohedra~\cite{FS,Post05} and the theory of the differential ring of polytopes, introduced by Buchstaber~\cite{B2008}, the previous result was generalized by Buchstaber and Limonchenko as follows.

\begin{theorem}[{\cite{BL}}]
For any given $\ell\geq 2$ and $n_{1},\ldots,n_{r}\geq 2, r\geq 1$ there exists a polyhedral product of the type $(\underline{D}^{2j_i},\underline{S}^{2j_{i}-1})^{K}=\mathcal Z_{P_{B}(J)}$, where $K:=K_{P_B}$ and $J=J(\ell,n_{1},\ldots,n_{r}):=(j_{1},\ldots,j_{m})$ for a certain flag nestohedron $P_B$ with $m$ vertices on a connected building set $B$, such that 
\begin{itemize}
\item The moment-angle manifold $\mathcal Z_{P_{B}(J)}$ is $\ell$-connected;
\item There exist strictly defined non-trivial Massey products of orders $n_{1},\ldots,n_{r}$ in $H^*(\mathcal Z_{P_{B}(J)})$.
\end{itemize}
\end{theorem}

Another way to generalize Theorem~\ref{mymainMassey} has been recently suggested by Grbi\'c and Linton. Their approach is based on a careful investigation, on the level of cochains, of cup and Massey products of the cohomology classes occured in Theorem~\ref{mainlemma}. Note that below the non-trivial Massey products, which are stated to exist, are no longer strictly defined, in general.

\begin{theorem}[{\cite{GL2}}]\label{mainGL}
The following statements hold.
\begin{itemize}
\item[(1)] For any given simplicial complexes $K_{1},\ldots,K_n$ there exists a sequence of stellar subdivisions of their join $K_{1}\ast\ldots\ast K_n$ resulting in a simplicial complex $K$ such that there exists a non-trivial $n$-fold Massey product in $H^*(\zk)$;
\item[(2)] If a simplicial complex $K^{\prime}$ is obtained from a simplicial complex $K$ by a sequence of a special type edge truncations and there exists a non-trivial $n$-fold Massey product in $H^*(\mathcal Z_{K^{\prime}})$, then the same property holds in $H^*(\zk)$.  
\end{itemize}
\end{theorem}

\begin{remark}
The case of $n=3$ in Theorem~\ref{mainGL} (1) coincides with the construction due to Baskakov~\cite{BaskM}. 
\end{remark}

Using Theorem~\ref{mainGL} (2) and Theorem~\ref{tripleMassey} (1), Grbi\'c and Linton obtain a result due to Zhuravleva~\cite{Zh}, who proved that for any {\emph{Pogorelov polytope}} $P$ (see~\cite{Pog,And,BE2017,BEMPP}) there exists a non-trivial triple Massey product in $H^*(\zp)$. 

Moreover, using Theorem~\ref{mainGL} (2) and Theorem~\ref{mymainMassey}, Grbi\'c and Linton showed that for any $n\geq 2$ there exists a non-trivial Massey product in $H^*(\mathcal Z_{Pe^n})$ of any order $k$, $2\leq k\leq n$, where $Pe^n$ is an $n$-dimensional permutohedron. 

It was earlier proved by the first author~\cite[Lemma 4.9, Theorem 4.10]{L2019} that if $P$ is a {\emph{graph-associahedron}}, in particular, an $n$-dimensional permutohedron, see~\cite{CD}, then the following conditions are equivalent: 1) $\zp$ is rationally formal; 2) there exist no non-trivial strictly defined triple Massey products in $H^*(\zp)$; 3) $P$ is a product of segments, pentagons, and hexagons.

Finally, it should be mentioned that there exists a simple polytope $P$ such that there are no non-trivial triple Massey products of 3-dimensional classes in $H^*(\zp)$, but there exists a non-trivial strictly defined 4-fold Massey product in $H^*(\zp)$. This was proved by Barali\'c, Grbi\'c, Limonchenko, and Vu\v{c}i\'c~\cite{BGLV} for $P$ being the dodecahedron using Theorem~\ref{mainlemma}. 

\section{Acknowledgements}
The authors wish to thank Victor Buchstaber for many fruitful discussions, encouragement, and advice. The first author is also grateful to the Fields Institute for Research in Mathematical Sciences, University of Toronto (Canada) for providing excellent research conditions and support while working on this paper at the Thematic Program on Toric Topology and Polyhedral Products.

\end{document}